\documentclass[11pt,a4paper, reqno]{article}
\usepackage[utf8]{inputenc}
\usepackage[T1]{fontenc}
\usepackage{amsfonts}
\usepackage{amssymb}
\usepackage{enumerate}
\usepackage{graphicx}
\usepackage{mathtools}
\usepackage{amsmath, amsthm}
\usepackage{tikz}
\usepackage{color}
\usepackage[affil-it]{authblk}
\usepackage[sort,numbers]{natbib}
\usepackage{bbm}
\usepackage[a4paper, total={6in, 9in}]{geometry}
\usepackage{amsmath}
\usepackage{amsfonts}
\usepackage{amssymb}
\usepackage{bbm}
\usepackage{mathrsfs}
\usepackage[utf8]{inputenc}
\usepackage[english]{babel}
\usepackage{hyperref}
\usepackage{relsize}
\usepackage{accents}
\usepackage{centernot}
\usepackage{subcaption}
\usepackage{url}

\usetikzlibrary{arrows,calc}

\makeatletter
  \@addtoreset{chapter}{part}
  \@addtoreset{@ppsaveapp}{part}
\makeatother

\long\def\/*#1*/{}

\newcommand\given[1][]{\:#1\vert\:}

\usepackage{palatino}

\makeatletter
\newsavebox\myboxA
\newsavebox\myboxB
\newlength\mylenA

\newcommand*\xoverline[2][0.75]{%
    \sbox{\myboxA}{$\m@th#2$}%
    \setbox\myboxB\null% Phantom box
    \ht\myboxB=\ht\myboxA%
    \dp\myboxB=\dp\myboxA%
    \wd\myboxB=#1\wd\myboxA% Scale phantom
    \sbox\myboxB{$\m@th\overline{\copy\myboxB}$}%  Overlined phantom
    \setlength\mylenA{\the\wd\myboxA}%   calc width diff
    \addtolength\mylenA{-\the\wd\myboxB}%
    \ifdim\wd\myboxB<\wd\myboxA%
       \rlap{\hskip 0.5\mylenA\usebox\myboxB}{\usebox\myboxA}%
    \else
        \hskip -0.5\mylenA\rlap{\usebox\myboxA}{\hskip 0.5\mylenA\usebox\myboxB}%
    \fi}
\makeatother

\numberwithin{equation}{section}

\usepackage{hyperref}
\hypersetup{
  colorlinks,
  linkcolor=blue,
  citecolor=red,
  filecolor=black,
  urlcolor=black,
  pdftitle={},
  pdfauthor={S. Dhara, R. van der Hofstad, J. van Leeuwaarden, S. Sen},
  pdfcreator={S. Dhara},
  pdfsubject={},
  pdfkeywords={}
}
\numberwithin{equation}{section}
\usepackage{cleveref}

\def\sss{\scriptscriptstyle}

\newcommand*{\Mtilde}[1]{\skew{5}{\tilde}{#1}}
\newcommand{\prob}[1]{\ensuremath{\mathbbm{P}\left(#1\right)}}
\newcommand{\expt}[1]{\ensuremath{\mathbbm{E}\left[#1\right]}}
\newcommand{\var}[1]{\ensuremath{\mathrm{Var}\left(#1\right)}}

\newcommand{\ind}[1]{\ensuremath{\mathbbm{1}{\left\{#1\right\}}}}
\newcommand{\pto}{\ensuremath{\xrightarrow{\mathbbm{P}}}}
\newcommand{\dto}{\ensuremath{\xrightarrow{d}}}

\newcommand{\PR}{\ensuremath{\mathbbm{P}}}
\newcommand{\E}{\ensuremath{\mathbbm{E}}}

\newcommand{\e}{\ensuremath{\mathrm{e}}}

\newcommand{\OP}{\ensuremath{O_{\sss\PR}}}
\newcommand{\oP}{\ensuremath{o_{\sss\PR}}}

\newcommand{\CM}{\ensuremath{\mathrm{CM}_n(\boldsymbol{d})}}
\newcommand{\ER}{\ensuremath{\mathrm{ER}_n(d/n)}}

\newcommand{\bld}[1]{\boldsymbol{#1}}
\newcommand{\C}{\mathscr{C}_{\sss (1)}}

\newcommand{\supth}{\mathrm{th}}
\newcommand{\Mcut}{{\sf{MaxCut}}}
\newcommand{\DB}{{\sf DistBip}}

\newcommand{\TC}{\mathrm{TC}}
\newcommand{\LTC}{\mathrm{LTC}}
\newcommand{\DS}{\mathrm{DS}}
\newcommand{\LDS}{\mathrm{LDS}}

\newcommand{\cG}{\mathcal{G}}
\newcommand{\sCL}{\mathscr{C}_{\sss (1)}}

\newtheorem{theorem}{Theorem}
\newtheorem*{claim*}{Claim}
\newtheorem{claim}{Claim}
\newtheorem{algo}{Algorithm}
\newtheorem{lemma}[theorem]{Lemma}
\newtheorem{proposition}[theorem]{Proposition}

\newtheorem{assumption}{Assumption}
\newtheorem{remark}{Remark}

\newtheorem{defn}{Definition}

\newenvironment{claimproof}{\begin{proof}\renewcommand{\qedsymbol}{\claimqed}}{\end{proof}\renewcommand{\qedsymbol}{\plainqed}}

\let\plainqed\qedsymbol

\numberwithin{equation}{section}
\numberwithin{theorem}{section}

\begin{document}

\title{Phase transitions of extremal cuts for \\
the configuration model}
\author[$1$]{Souvik Dhara\footnote{s.dhara@tue.nl \hspace{3cm}$^\dagger$d.mukherjee@tue.nl \hspace{2.5cm}$^\ddagger$subse@microsoft.com}}
\author[$1$]{Debankur Mukherjee$^\dagger$}
\author[$2$]{Subhabrata Sen$^\ddagger$}
\affil[$1$]{%Department of Mathematics and Computer Science,
Eindhoven University of Technology, The Netherlands}
\affil[$2$]{Microsoft Research New England
and MIT}

\renewcommand\Authands{, }

\date{\today}

\maketitle 

\begin{abstract} 
The $k$-section width and the Max-Cut for the configuration model are shown to exhibit phase transitions according to the values of certain parameters of the asymptotic degree distribution. These transitions mirror those observed on Erd\H{o}s-R\'enyi random graphs, established by Luczak and McDiarmid (2001), and Coppersmith et al.~(2004), respectively. 

%\textcolor{blue}{ Our proofs crucially utilize a novel sequential construction of configuration model random graphs.} 
\end{abstract}

\section{Introduction} 
Graph cut problems have a very rich history in Combinatorics and Theoretical Computer Science. 
Given a graph $G= (V,E)$, the $k$-section problem seeks to partition the vertices $V= V_1 \sqcup V_2 \sqcup \cdots \sqcup V_k$ into $k$ equal sets (or differing by at most 1) such that the number of edges between the distinct sets is minimized. 
The minimum number of cross edges $w_k(G)$ thus obtained is referred to as the $k$-section  width.
The related Max-Cut problem seeks to divide the vertices into two sets (not necessarily equal) such that the number of edges between the two sets is maximized. 
{These graph-partitioning problems are extremely important for numerous practical applications  in network optimization, VLSI circuit design, computational geometry, and statistical physics \cite{CKC83,CD87,HM02,MPV87,JS93,poljak1995tuza, diazpetitserna}.} On the other hand, from the perspective of Theoretical Computer Science, these problems are computationally hard, and even approximating the Max-Cut up to a constant factor is NP-hard~\cite{FK02, K04, HZ01,H01}. 
The study of these problems in the average case is mainly motivated by a desire to understand various graph partitioning heuristics. 
Problem instances are usually chosen to be the Erd\H{o}s-R\'enyi random graph, or the random regular graph. 
An  Erd\H{o}s-R\'enyi random graph $\mathrm{ER}_n(d/n)$ is constructed on $n$ vertices, where any two vertices share an edge with probability $d/n$, independently of each other. 
A $d$-regular random graph is drawn uniformly at random from the space of all $d$-regular graphs on $n$ vertices. 
We note that these graph ensembles are \emph{sparse}, in that typical graphs on $n$ vertices have order $n$ edges and the degree of a typical vertex is of the constant order. 
{See~\cite{bollobas,RGCN1,RGCN2,JLR00}} for a detailed review of the properties of these random graphs.

Both $k$-section width and Max-Cut undergo phase transitions on the sparse Erd\H{o}s-R\'enyi random graph. 
These transitions reflect certain structural characteristics of the underlying graphs. 
Consider the $k$-section width problem for $\mathrm{ER}_n(d/n)$, with $k=2$. 
For $d< 2\ln (2)$, the bisection width is exactly $0$ with high probability, while for $d > 2 \ln( 2)$, the bisection width is of order $n$, with high probability \cite{LM01}. 
The Max-Cut also undergoes a phase transition;
 for $d<1$, the difference between the total number of edges and the Max-Cut is of the constant order, while it is of the order $n$ for $d>1$~\cite{CGHS04}. 
The distribution of the Max-Cut within the critical window is analyzed by~\citet{DMRR12}, while the critical behavior of the bisection width is largely unknown.

A crucial point to note in this context is that both sparse Erd\H{o}s-R\'enyi and random regular graph ensembles lead to homogeneous instances, in the sense that any two vertices share an edge with equal probability. 
This is very different from the instances actually encountered in practical applications. 
Real networks are extremely inhomogeneous, and often display certain characteristic features, such as  a power-law decay in the tail of the degree distribution {\cite{SFFF03,FFF99,BA99,AB02,RGCN1}}. 
Thus, it is of natural interest to study the behavior of the \emph{extremal} cuts for graphs with more general degree distributions. 
The configuration model~\cite{B80,MR95} provides a canonical scheme for generating uniform random graphs with \emph{any} prescribed degree sequence. 
This model is thus attractive for studying real-world networks, and analysis of its structural properties have attracted considerable attention in recent years \cite{MR95,J08,J09,J09b,JL07,JL09}. 
It is worthwhile to mention that despite the presence of very high degree vertices, a plethora of modern research remarkably conveys a qualitatively similar behavior of various statistics in this model to those in Erd\H{o}s-R\'enyi random graphs, confirming empirical evidences.

In this paper, we initiate a study of similar phase transition phenomena of the extremal cuts for the configuration model.
The main takeaway of our results is that the phase transitions for the extremal cuts are robust, and are present in a large class of random graphs, viz.~configuration models with finite second moment. 
This emphasizes that in the class of sparse non-spatial random graphs these phase transition phenomena are not intimately dependent on the precise model details, but are determined by the component sizes and the structures of the typical local neighborhoods.  
Technically, the proofs in the Erd\H{o}s-R\'enyi case crucially utilize the independence and homogeneity in the model --- while 
we rely on the recent insights about the structure of the configuration model~\cite{J08, J09b, JL07,JL09}
to establish our results.
We also prove several novel structural properties of the connected components (see Sections~\ref{ssec:ksection-supcrit} and~\ref{sec:local_approx}).
Among many other intermediate results, we show that the largest connected component consists of a \emph{well-connected} 2-core (Lemma~\ref{lem-epsilon-delta-2-core}) and several \emph{thin} hanging trees (Lemma~\ref{lem:small-hanging-trees}), and most of the connected components except the largest are finite (Lemma~\ref{lem:small-no-giant}).
Furthermore, we obtain that when the largest connected component is of order $n$, it must be \emph{stable}, in the sense that $\Theta(n)$ edges \emph{must} be deleted in order to separate out any $\Theta(n)$ vertices (Proposition~\ref{prop-no-e-d-cut}).
The latter notion is particularly useful to study the stability of the largest connected component subject to intelligent attacks (edge deletion) on networks.

The rest of the paper is organized as follows: Section~\ref{sec:model} formally introduces the configuration model along with the assumptions on the underlying degree sequence and summarizes certain preliminary properties of this model. 
Section~\ref{sec:main_results} states the main results of this paper 
and offers several key insights. 
The proofs are included in Sections~\ref{sec:proof_minbisection}~and~\ref{sec:proof_maxcut}.

\section{Preliminaries}
\label{sec:model}

\paragraph{The configuration model.}
Consider a degree sequence $\bld{d}=(d_1,d_2,\dots,d_n)$ on the vertex set $[n]=\{1,2,\dots,n\}$. 
Equip vertex $j$ with $d_{j}$ stubs or \emph{half-edges}. Two half-edges create an edge once they are paired. 
Therefore, initially there are $\ell_n=\sum_{i \in [n]}d_i$ half-edges. 
Pick any one half-edge and pair it with a uniformly chosen half-edge from the remaining unpaired half-edges. 
Keep repeating the above procedure until all the unpaired half-edges are exhausted. 
The random graph constructed in this way is called the configuration model, and will henceforth be denoted by $\mathrm{CM}_{n}(\boldsymbol{d})$.
Moreover, under rather general assumptions (see Assumption~\ref{assumption1} below), the asymptotic probability of the graph being simple is bounded away from zero~\cite{J09c}.

  Note that the graph constructed by the above procedure may contain self-loops and multiple edges. It can be shown that conditionally on $\mathrm{CM}_{n}(\boldsymbol{d})$ being simple, the law of such graphs is uniform over all possible simple graphs with degree sequence $\boldsymbol{d}$ (cf.~\cite[Proposition 7.7]{RGCN1}, \cite{JKLP93}). 
  
A vertex chosen uniformly at random from the vertex set $[n]$, independently of the graph $\CM$ is called a \emph{typical} vertex.
  Let $D_n$ be the degree of a typical vertex. 
Throughout this paper we assume the following:
\begin{assumption}\label{assumption1}
\normalfont 
Let $\bld{d} = \bld{d}_n$ be a degree sequence on $[n]$. 
The sequence of degree sequences $(\bld{d}_n)_{n\geq 1}$ is such that
\begin{enumerate}[a.] 
\item \label{assumption1-1}  $D_n\dto D$ (\emph{weak convergence of the degree of a typical vertex});
\item \label{assumption1-2}  $\E[D_n]\to \E[D]$, and $\E[D_n^2]\to\E[D^2]$ (\emph{moment assumptions});
\item \label{assumption1-3} $\PR(D=1)>0$ (\emph{positive proportion of degree one vertices}).
\end{enumerate}
\end{assumption}

Like most other sparse random graph models, $\CM$ exhibits a phase transition in terms of the size of its largest connected component, and this has been studied extensively in \cite{MR95,JL09}. 
The phase transition occurs when the value of the parameter 
\begin{equation}
 \nu : = \frac{\expt{D(D-1)}}{\expt{D}} \label{eq:nu}
\end{equation}
exceeds one (cf.~\cite{JL09}).
More precisely, let $g_D(x):=\E[x^D]$ be the probability generating function of $D$, and let $\xi$ be the unique nonzero solution to the equation $g_D'(x)=\E[D] x$. Define
\begin{equation}\label{defn:eta}
\eta=1-g_D(\xi).
\end{equation}  
For  $i\geq 1$, denote the $i^{\supth}$ largest component of $\CM$ by $\mathscr{C}_{\sss (i)}$. 
Then the following theorem characterizes the asymptotic proportion of vertices in each component:
\begin{theorem}[{\cite[Theorem 2.3]{JL09}}]\label{thm:JL09}
Consider \CM\ satisfying {\rm Assumption~\ref{assumption1}}. Then, 
\begin{enumerate}[{\normalfont (i)}]
 \item $|\mathscr{C}_{\sss (1)}|/n\pto \eta$, as $n\to\infty$, where $\eta$ is as defined in \eqref{defn:eta}. 
 Further, $\eta >0$ if and only if $\nu>1.$
 \item Moreover, $|\mathscr{C}_{\sss (i)}|/n\pto 0$, as $n\to\infty$, for all $i\geq 2$.
\end{enumerate}
\end{theorem}
\paragraph{Notation.} 
For any graph $G$, the $k$-section width and Max-Cut are denoted by $w_k(G)$ and $\Mcut(G)$, respectively.
We denote 
$$
\mu=\expt{D}
$$
to be the (asymptotic) expected degree of a typical vertex.
%We denote the number of edges in the graph $G$ by $\mathrm{E}(G)$
The degree of a vertex $v$ is denoted by $d_v$, and 
the number of vertices of degree $k$ by $n_k$, $k\geq 0$.
If two vertices $u$ and $v$ share an edge, then we write $u\leftrightsquigarrow v$.
For a nonempty subset $U\subseteq [n]$ of vertices, the neighborhood (or 1-neighborhood) is defined as 
$$\mathcal{N}[U,1]:= U\cup \{v\in[n]:u\leftrightsquigarrow v\mbox{ for some }u\in U\},$$
and the $r$-neighborhood is defined as $\mathcal{N}[U,r]:=\mathcal{N}[\mathcal{N}[U,r-1],1]$, $r>1$.
For any subset of vertices $A$, we denote the half-edges incident to the vertices in $A$ by $S(A)$, and the number of edges between $A$ and $A^c$ by $E(A,A^c)$.
For any integer $m\geq 1$, we denote $(2m)!! := (2m-1)(2m-3)\cdots 1$.
All the limiting statements should be understood as $n\to\infty$, unless specified otherwise. 
For a sequence of probability measures $(\PR_n)_{n\geq 1}$, the sequence of events $(\mathcal{E}_n)_{n\geq 1}$ is said to hold \emph{with high probability} if $\PR_n(\mathcal{E}_n)\to 1$.
We use the usual Bachmann-Landau notations $o(\cdot)$, $O(\cdot)$, and $\Theta(\cdot)$ to write asymptotic comparisons.
For two sequences of random variables $(X_n)_{n\geq 1}$, and $(Y_n)_{n\geq 1}$, we write $X_n = \oP(Y_n)$ to denote that $X_n/Y_n\xrightarrow{\sss \PR} 0$.

\section{Main results}
\label{sec:main_results}
In this section we state the main results of this paper, and discuss several heuristics.
\begin{theorem}[Phase transition of the $k$-section width]\label{thm:min-bis}
Consider $\CM$ satisfying {\rm Assumption~\ref{assumption1}}, and let $k\geq 2$ be an integer.
 Then $w_k(\CM)$ exhibits a phase transition around $\eta =1/k$. More precisely, 
 \begin{enumerate}[{\normalfont (i)}]
 \item \label{thm:min-bis-i} If $\eta <1/k$, then  with high probability $w_k(\CM)\leq k/2$;
 \item \label{thm:min-bis-ii} If $\eta >1/k$, then there exists $\zeta>0$, such that with high probability $w_k(\CM)>\zeta n$.
 \end{enumerate}
\end{theorem}
Theorem \ref{thm:min-bis} is proved in Section \ref{sec:proof_minbisection}. This result is comparable to \cite[Theorem 1]{LM01}, established in the context of Erd\H{o}s-R\'enyi random graphs. 
As mentioned earlier, the proof for the Erd\H{o}s-R\'enyi case makes crucial use of the fact that the edge occupancies are independent and identically distributed --- a feature that is absent in this case. 
The proof in this paper, on the other hand, is more robust, and depends on a clear understanding of the local neighborhood structure in these random graphs. 
Roughly speaking, when $\eta<1/k$, the strategy is to distribute all the components of size at least 3 among $k$ partitions as evenly as possible, and then to add the components of size at most 2 to balance the partitions.
Since  the size of the largest component is smaller than $n/k$ and the other components are very small ($o(n)$) in size, a $k$-partition  can be made using the components of size at least 3, with at most $n/k$ vertices in each part.
Because there are sufficiently many components of size at most~2 (Lemma~\ref{lem:num-pairs}),
these can be used to balance the partitions. 
The latter step results in at most $k/2$ cross edges between the partitions. 
The above proof outline for the subcritical case is formalized in Section~\ref{ssec:ksection-subcrit}. 
Alternatively, when $\eta>1/k$, the size of the largest connected component is more than $n/k$. 
Therefore, in order to split the graph into $k$ equal partitions,
the largest component must be split into at least two (possibly unequal) parts, each containing a positive proportion of vertices, and from the structural properties of the largest component, we show that with high probability this creates $\Theta(n)$ cross edges. 
The proof for the supercritical case is provided in Section~\ref{ssec:ksection-supcrit}.
\begin{remark}\label{rem:isolated-vert}\normalfont
\cite[Theorem 1]{LM01} establishes that the $k$-section width is exactly zero below a critical threshold given by $\eta = 1/k$. 
This holds for the Erd\H{o}s-R\'enyi case due to the natural presence of many isolated vertices. 
For a general configuration model, this is not necessarily true, and therefore, Theorem \ref{thm:min-bis}~(i) is indeed the best possible result that one can hope for in this case. 
In particular, if we assume the presence of a positive fraction of isolated vertices in the degree sequence, then using Lemma~\ref{lem:graph-sub} below, we recover the same result as in~\cite{LM01}. 
\end{remark}

We continue to describe our results for $\Mcut(\CM)$.
For this let us introduce a further notation. 
The difference between the total number of edges and the Max-Cut is often referred to as the \emph{distance from bipartiteness} of a graph $G$, and will be denoted by $\DB(G)$.
In other words, $\DB(G)$ counts the minimum number of edges in $G$ to be deleted in order to make it bipartite. 
Recall that $\mu = \expt{D}$.
\begin{theorem}[Phase transition of the Max-Cut]
\label{thm:max-cut}
 Consider $\CM$ satisfying {\rm Assumption~\ref{assumption1}}. Then $\Mcut(\CM)$ admits a phase transition around $\nu=1$. More precisely, 
  \begin{enumerate}[{\normalfont (i)}] 
  \item {\normalfont (Subcritical)} If $\nu<1$, then as $n\to\infty$,
  $$\DB(\CM)\overset{\mathrm{d}}{\longrightarrow} Z\sim\mathrm{Poisson}\left(\frac{1}{4}\ln\left(\frac{1+\nu}{1-\nu}\right)\right).$$
  \item {\normalfont(Supercritical)} If $\nu>1$, then there exists $\delta>0$, such that with high probability, 
  $$\DB(\CM)>\delta n.$$
  \item {\normalfont(High-density regime)} Furthermore, when $\mu>2$, then there exists $0<c^\star(\mu)<\sqrt{\mu} /4$, such that for any $c>c^{\star}(\mu)$, with high probability, 
  $$\Mcut(\CM) \leq n\left(\frac{\mu}{4} +  c \sqrt{\mu}\right),$$ 
 and $c^\star(\mu)\nearrow \sqrt{\ln (2)}/2$  as $\mu\nearrow\infty$.
   \end{enumerate}
\end{theorem}
The proof of Theorem \ref{thm:max-cut} is included in Section \ref{sec:proof_maxcut}. 
Theorem \ref{thm:max-cut} establishes the phase transition for $\DB(\CM)$ for a wide class of degree sequences. 
The heuristic behind this phase transition is that when $\nu<1$, $\CM$ is \emph{roughly} a collection of trees and a \emph{finite} number of unicyclic components.
The trees do not contribute any edge to $\DB(\CM)$ at all, and the unicyclic components with an odd cycle (i.e., containing an odd number of edges) contributes at most one to the $\DB(\CM)$. 
On the other hand, when $\nu>1$, this is no longer true, and any partition must leave $\Theta(n)$ edges uncut.
Results analogous to Theorem~\ref{thm:max-cut}~(i)~and~(ii) were established for Erd\H{o}s-R\'enyi random graphs by Daud\'e at al.~\cite{DMRR12} and Coppersmith et al.~\cite{CGHS04}, respectively.  
\begin{remark}\label{rem:simple}
\normalfont
It was shown in \cite{J09c} that under Assumption~\ref{assumption1}, the probability of the graph being simple is bounded away from zero. 
Thus the phase transition results in Theorems~\ref{thm:min-bis}~and~\ref{thm:max-cut} also hold for the uniformly chosen simple graph with a prescribed degree sequence. 
Hence, all the results proved in the paper are true also for $\ER$, as well as the generalized random graphs under appropriate conditions \cite[Theorem 6.15]{RGCN1} on the weight sequence $\boldsymbol{w}$. 
In fact, the results are true for an even more general class of inhomogeneous random graph models (cf. \cite[Theorem 6.18]{RGCN1}).
\end{remark}
\begin{remark} 
\normalfont 
%As we will see in the proof of Theorem~\ref{thm:max-cut}, the value of $c^*(\mu)$ is given by
%$$c^\star(\mu)=\inf_{c>0}\bigg\{c: \Big(\frac{1}{4}-\frac{c}{\sqrt{\mu}}\Big)^{-\left(\frac{1}{4}-\frac{c}{\sqrt{\mu}}\right)}\Big(\frac{1}{4}+\frac{c}{\sqrt{\mu}}\Big)^{-\left(\frac{1}{4}+\frac{c}{\sqrt{\mu}}\right)}<2^{-\frac{1}{\mu}+1} \bigg\}.$$
\begin{figure}
  \centering
  \includegraphics[scale=.55]{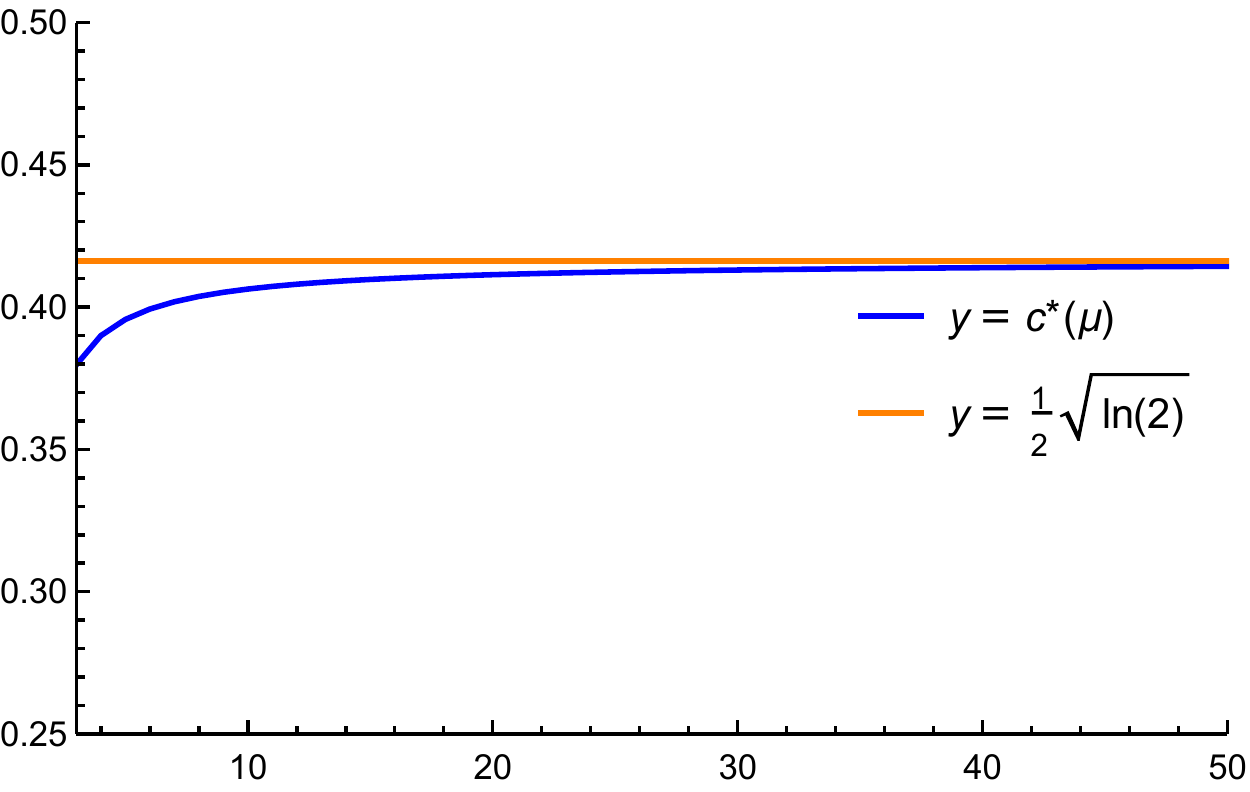}
\caption{Numerical values of $c^\star(\mu)$ for  $3\leq \mu \leq 50$.}
\label{fig:threshold-c}
\end{figure}
\noindent
%Numerically it can be seen that $c^\star(\mu)$ increases with $\mu$, and is approximately equal to $\sqrt{\ln(2)}/2$ for large enough $\mu$.
Figure~\ref{fig:threshold-c} shows the numerical values of $c^\star(\mu)$ for $3\leq \mu \leq 50$.
An exact expression of $c^\star(\mu)$ is given in~\eqref{eq:cstarmu}.
Notice that even for $\mu$-values as low as~30, $c^\star(\mu)$ is sufficiently close to $\sqrt{\ln(2)}/2$. 
This value agrees with the upper bound of Max-Cut for Erd\H{o}s-R\'enyi random graphs and random regular graphs in the high density regime as observed in~\cite[Theorem 20]{CGHS04} and~\cite[Theorem 2]{BCP97}, respectively.
Thus,  our result again establishes a universal behavior for a large class of inhomogeneous random graphs (see Remark~\ref{rem:simple}) as special cases.
\end{remark}

To further illustrate the usefulness of the above phase transition results, we consider graphs obtained by random deletion of edges from a given graph. Such results are crucial for studying the stability of networks to random link failures. 
Percolation refers to keeping the edges of a graph with a given probability $p_n$, independently among each other and independent of the underlying (random) graph. 
Using Theorems~\ref{thm:min-bis} and~\ref{thm:max-cut} we are able to characterize the threshold of the percolation probability for the configuration model, with respect to the $k$-section width and the Max-Cut.
Let $\mathrm{CM}_n(\bld{d},p_n)$ be the graph obtained by retaining the edges of $\CM$ with probability $p_n$. 
An important property of $\CM$ is that $\mathrm{CM}_n(\bld{d},p_n)$ is again distributed as a configuration model conditionally on its degree sequence \cite{F07,J09}. 
Therefore, one can deduce the phase transition results for the extremal cuts of $\mathrm{CM}_n(\bld{d},p_n)$ from Theorems~\ref{thm:min-bis}~and~\ref{thm:max-cut}. 
In fact, since the percolated graphs always have a positive proportion of isolated vertices in the sparse regime (Assumption~\ref{assumption1}), the minimum bisection below the threshold $\eta =1/k$  becomes exactly zero with high probability (see Remark~\ref{rem:isolated-vert}).

Let $k\geq 2$ be an integer.
Then the phase transition for $w_k(\mathrm{CM}_n(\bld{d},p_n))$ with $p_n\to p$, occurs at $p=p_{\min}(k,\bld{d})$, such that the asymptotic proportion of vertices in the largest connected component of $\mathrm{CM}_n(\bld{d},p_n)$ is precisely equal to $1/k$. 
For an arbitrary degree sequence, the explicit solution for $p_{\min}(k,\bld{d})$ is not immediate from \cite[Theorem 3.9]{J09}. 
However, in the particular case of percolation on the $d$-regular graph (i.e. $\bld{d} = d\bld{1} = (d,d,\dots,d)$) with $d\geq 3$, notice that by \cite[(3.13),(3.14)]{J09}, $p_{\min}(k,d)$ can be obtained as a solution for $p$ in the following system of equations:
\begin{equation}
 \sqrt{p}d(1-\sqrt{p}+\sqrt{p}\xi)^{d-1}+(1-\sqrt{p})d = d\xi, \quad 1-(1-\sqrt{p}+\sqrt{p}\xi)^{d}=\frac{1}{k},
\end{equation}and thus,
\begin{equation}
 p_{\min}(k,d) := \frac{1-\left(1-\frac{1}{k}\right)^{\frac{1}{d}}}{1-\left(1-\frac{1}{k}\right)^{\frac{d-1}{d}}}.
\end{equation}

It was shown in \cite[Theorem 3.9]{J09} that, when $p_n\to p$, the phase transition for the largest connected component occurs at $p=1/\nu$. 
This implies that the phase transition for $\Mcut(\mathrm{CM}_n(\bld{d},p_n))$ also occurs at $1/\nu$, which for $d$-regular random graphs equals
\begin{equation}
p_{\max}(d):=\frac{1}{d-1}.
\end{equation}

Therefore, given the phase transition results in Theorems~\ref{thm:min-bis} and~\ref{thm:max-cut}, we have proved the following theorem:
\begin{theorem}[{Extremal cuts for percolation on random $d$-regular graphs}]
Let $p_n\to p$ as $n\to\infty$.
Then for any $d\geq 3$, 
\begin{enumerate}[{\normalfont (i)}]
\item 
\begin{enumerate}[{\normalfont (a)}]
\item If $p<p_{\min}(k,d)$, then with high probability, $w_k(\mathrm{CM}_n(d\bld{1},p_n))= 0$.
\item Furthermore, if $p>p_{\min}(k,d)$, then there exists $\zeta>0$, such that with high probability, $w_k(\mathrm{CM}_n(d\bld{1},p_n))>\zeta n$.
\end{enumerate}
\item 
\begin{enumerate}[{\normalfont (a)}]
\item If $p<p_{\max}(d)$, then 
  $$\DB(\mathrm{CM}_n(d\bld{1},p_n))\overset{\mathrm{d}}{\longrightarrow} Z\sim\mathrm{Poisson}\left(\frac{1}{4}\ln\left(\frac{1+(d-1)p}{1-(d-1)p}\right)\right).$$
\item If $p>p_{\max}(d)$, then there exists $\delta>0$, such that with high probability,
 $$\DB(\mathrm{CM}_n(d\bld{1},p_n))>\delta n,$$
\item Further, if $p>2/d$, then for any $c> c^{\star}(d p)$, 
  $$\Mcut(\mathrm{CM}_n(d\bld{1},p_n)) \leq n\left(\frac{d p}{4} + c \sqrt{d p}\right),$$
 with high probability, where $c^{\star}(\cdot)$ is as given by Theorem~\ref{thm:max-cut}. 
  %where $c^\star(\mu)\approx \ln (2)$ for large $\mu$.
\end{enumerate}
\end{enumerate}
\end{theorem}

\section{{Proof for the \texorpdfstring{$\bld{k}$}{k}-section width}}
\label{sec:proof_minbisection}
In this section we prove the phase transition of the $k$-section width stated in Theorem~\ref{thm:min-bis}. 
\subsection{Subcritical case} 
\label{ssec:ksection-subcrit}
In this subsection we present the proof of Theorem~\ref{thm:min-bis} (i).
In Lemma~\ref{lem:graph-sub} we first state a useful graph theoretic result, which ensures that if (i) the size of the largest component is smaller than $n/k$,
(ii) there are $\Theta(n)$ \emph{small} components (i.e., of size at most 2), and
 (iii) the size  of every component other than the $k$ largest components is smaller than the $k^{\mathrm{th}}$ fraction of the number of small components,    then the $k$-section width is at most $k/2$.
 This lemma is an extension of~{\cite[Lemma 9]{LM01}} to fit in the scenario when there are possibly no isolated vertices.
Then in Lemma~\ref{lem:num-pairs} we show that under Assumption~\ref{assumption1}, $\Theta(n)$ such \emph{small} components are created.
This will complete the proof of Theorem~\ref{thm:min-bis} (i).
\begin{lemma} 
\label{lem:graph-sub}
Consider a graph $G$ on $n$ vertices, with $m$ components of sizes $c_1\geq \dots\geq c_m$ such that 
(i) $c_1\leq n/k$, 
(ii) $\#\{i:c_i\leq 2\}\geq rn$ for some $r>0$, and
(iii) $c_i\leq rn/k$ for all $i>k$. 
Then, $w_k(G)\leq k/2$.
In addition, if $\#\{i:c_i=1\}\geq k-1$, then  $w_k(G)=0.$
\end{lemma}
\begin{proof}
Suppose that  $G$ contains $m_{ 2}$ components of size more than 2, and
enumerate them as $C_1$, $C_2$, $\ldots, C_{m_{2}}$ with sizes $c_1\geq \dots\geq c_{m_2},$ respectively (ties can be broken arbitrarily).
We construct $k$ partitions $V_1, V_2,\ldots, V_k$ sequentially as follows.
Define $V_1(1)=C_1$, and $V_i(1)=\emptyset$ for $i=1,\ldots,k$. 
For $2\leq t\leq m_2$,
$$V_i(t) = 
\begin{cases}
V_i(t-1)\cup C_t\quad\text{if}\quad i = \min\{j:|V_j(t-1)\cup C_t|\leq n/k\},\\
V_i(t-1)\quad\qquad\text{otherwise},
\end{cases}$$
i.e., sequentially at each step add all the vertices in components of size more than 2, to the partitions in a way such that the size of each partition does not exceed $n/k$.
The claim below establishes that the above steps are feasible.
\begin{claim}
For all $2\leq t\leq m_2$, $\min\{j:|V_j(t-1)\cup C_t|\leq n/k\}\leq k$.
\end{claim}
\begin{claimproof}
Note that, due to condition~(i), $|V_i(1)|\leq n/k$.
Now, if possible assume that at step $t_0\leq m_2$, $|V_j(t_0-1)\cup C_{t_0}|= |V_j(t_0-1)|+| C_{t_0}|>n/k$ for all $j=1,\ldots,k$.
Summing over $j$, we have 
\begin{align*}
\sum_{j=1}^k|V_j(t_0-1)|+k| C_{t_0}|>n
\implies  | C_{t_0}|>\frac{n}{k}-\frac{1}{k}\sum_{j=1}^k|V_j(t_0-1)|\geq \frac{n}{k} - \frac{n-rn}{k},
\end{align*}
due to condition (ii). This in turn implies $|C_{t_0}|>\frac{rn}{k},$
which contradicts condition~(iii).
\end{claimproof}  
After step $t=m_2$, we first add the components of size~2 and finally components of size~1 (the isolated vertices), if any.
Observe that  components of size~1,~2 can be added to the partitions such that each partition is of size $\lfloor n/k\rfloor$ or $\lfloor n/k\rfloor+1$, there are no cross edges between the partitions, and the number of vertices remaining to be included in any partition is at most $k-1$.
Now, if  $\#\{i:c_i = 1\} \geq k-1$, then at the last step
the remaining vertices must be isolated ones, and these do not create any cross edge, and thus the $k$-section width is exactly zero. 
Otherwise, the remaining $k-1$ vertices can form at most $k/2$ cross edges (the worst case being there are no isolated vertices).
\end{proof}

We will now verify that $\CM$ with $\eta<1/k$ satisfies all the conditions of Lemma~\ref{lem:graph-sub}, with high probability.
Condition~(i) follows from Theorem~\ref{thm:JL09}~(i) and the fact that $\eta<1/k$. 
In Lemma~\ref{lem:num-pairs} below, we will show that the number of components of size 2 scaled by~$n$, converges in probability to a positive constant, which verifies Condition~(ii).
Finally, Condition~(iii) is a consequence of Theorem~\ref{thm:JL09}~(ii).
The proof of Theorem~\ref{thm:min-bis}~\eqref{thm:min-bis-i} is now complete by Lemma~\ref{lem:graph-sub}.
\qed\\

 Recall that $n_1$ denotes the number of vertices in $\CM$ with degree one, and $n_1/n\to \PR(D=1) = p_1>0$. Suppose that the degree one vertices are indexed as $1,2,\dots,n_1$.
 We say that a \emph{pair} is created if a degree one vertex is joined with another degree one vertex. 
 Thus, the pairs are the components of size 2 in $\CM$.
\begin{lemma}
\label{lem:num-pairs} 
Let $P_n:=\sum_{1\leq i<j \leq n_1}\ind{i\leftrightsquigarrow j}$ be the number of pairs in \CM. Then, as $n\to\infty$, 
$$\frac{P_n}{n}\pto \frac{p_1^2}{2\E[D]}.$$
\end{lemma}
\begin{proof}
 Note that, by Assumption~\ref{assumption1}
 \begin{equation}
  \frac{1}{n}\expt{P_n}= \frac{1}{n}\sum_{1\leq i<j\leq n_1}\prob{i\leftrightsquigarrow j} = \frac{1}{n}\binom{n_1}{2}\frac{1}{\ell_n-1}\to \frac{p_1^2}{2\E[D]}.
 \end{equation} Further, if $I=\{(i_1,j_1,i_2,j_2): 1\leq i_1<j_1\leq n_1, 1\leq i_2<j_2\leq n_1, i_1,i_2,j_1,j_2 \text{ are distinct}\}$, then
 \begin{equation}
 \begin{split}
  &\frac{1}{n^2}\expt{P_n^2} = \frac{1}{n^2}\bigg(\sum_{i_1,j_1,i_2,j_2\in I}\prob{i_1\leftrightsquigarrow j_1, i_2\leftrightsquigarrow j_2} +\sum_{1\leq i<j\leq n_1}\prob{i\leftrightsquigarrow j} \bigg)\\
  &= \frac{1}{n^2}\bigg(\frac{1}{(\ell_n-1)(\ell_n-3)}\binom{n_1}{2}\binom{n_1-2}{2}+\binom{n_1}{2}\frac{1}{\ell_n-1}\bigg)\longrightarrow \bigg(\frac{p_1^2}{2\E[D]}\bigg)^2.
  \end{split}
 \end{equation}
 Therefore,
 \begin{equation}
  \frac{1}{n^2}\var{P_n}\to 0,
 \end{equation} 
 and an application of Chebyshev's inequality completes the proof.
\end{proof}

\subsection{Supercritical case}
\label{ssec:ksection-supcrit}
In this subsection we prove the supercritical case of the $k$-section width stated in Theorem~\ref{thm:min-bis}~(ii).
As mentioned earlier, since $\eta>1/k$, the fraction of vertices in the largest component is more than $1/k$, with high probability.
Therefore, in any balanced $k$-partition of the graph $G$, there must exist two distinct partitions each containing an asymptotically positive proportion of vertices from the largest component.
It is thus enough to show that if the largest component is partitioned into two sets $V_1$, $V_2$, each containing a positive proportion of vertices, then  with high probability, there exist $\Theta(n)$ cross-edges between $V_1$ and $V_2$.
 The following key definition formalizes this cut-property:
\begin{defn}[$\varepsilon$-$\delta$ cut]\normalfont Given $\varepsilon, \delta>0$, an $(\varepsilon, \delta)$-cut of a graph $G=(V,E)$ is a partition of $V$ in two sets $V_1$, and $V_2$ such that $|V_1|,|V_2|> \varepsilon|V|$, and the number of edges between $V_1$ and $V_2$ is at most $\delta|V|$.
\end{defn} 
Now observe that the following proposition is enough to conclude Theorem~\ref{thm:min-bis} (ii):
\begin{proposition}\label{prop-no-e-d-cut} 
Consider $\CM$ with $\nu >1$ and satisfying {\rm Assumption~\ref{assumption1}}. For any $\varepsilon > 0$, there exists $\delta= \delta(\varepsilon) >0$ such that with high probability the giant component $\mathscr{C}_{\sss (1)}$ does not have an $(\varepsilon,\delta)$-cut. 
\end{proposition}

We now briefly sketch the outline of the  proof of Proposition~\ref{prop-no-e-d-cut}.
The idea was first introduced by Bollob\'as et al.~\cite{BJR07} in the context of stability of the largest connected component of inhomogeneous random graphs.
We leverage their technique for the configuration model, and in conjunction with suitable structural properties of the giant component, prove Proposition~\ref{prop-no-e-d-cut}.
The application of this technique to the configuration model poses substantial challenge due to the dependence among edges, and the methods for inhomogeneous random graphs~\cite{BJR07} or Erd\H{o}s-R\'enyi random graphs~\cite{LM01} are not directly applicable.
In this paper, we therefore present some novel arguments 
that establish the necessary structural properties for this proof technique to work.
In particular, we introduce a sequential construction
of the configuration model in Subsection~\ref{sssec:notheavy}, 
that facilitates the comparison between $\CM$ and the graph with one deleted vertex.

For any graph $G$ with vertex set $V$, define the \emph{$k$-core} to be the maximal set of vertices $V^{k}\subseteq V$, such that in the subgraph induced by $V^{k}$, each vertex has degree at least~$k$. 
Note that the $k$-core of any graph is unique, although it can possibly consist of an empty graph only.
It is worthwhile to note that the 2-core of any connected graph is also connected.
Algorithmically, the $k$-core of a graph can be obtained by sequentially deleting the vertices of degree less than $k$
along with all their incident edges, until all the vertices in the remaining graph have degree more than $k$.
Observe that, $V^{k}\supseteq V^{k+1}$, and the subgraph induced by $V\setminus V^{2}$ is a forest.
See Figure~\ref{fig:sfiga} for an instance of the 2-core of a graph and the trees hanging from it. 
Figure~\ref{fig:sfigb} visualizes the 3-core as a subset of the 2-core.
\begin{figure}
\begin{subfigure}{.5\textwidth}
  \centering
  \includegraphics[scale=.45]{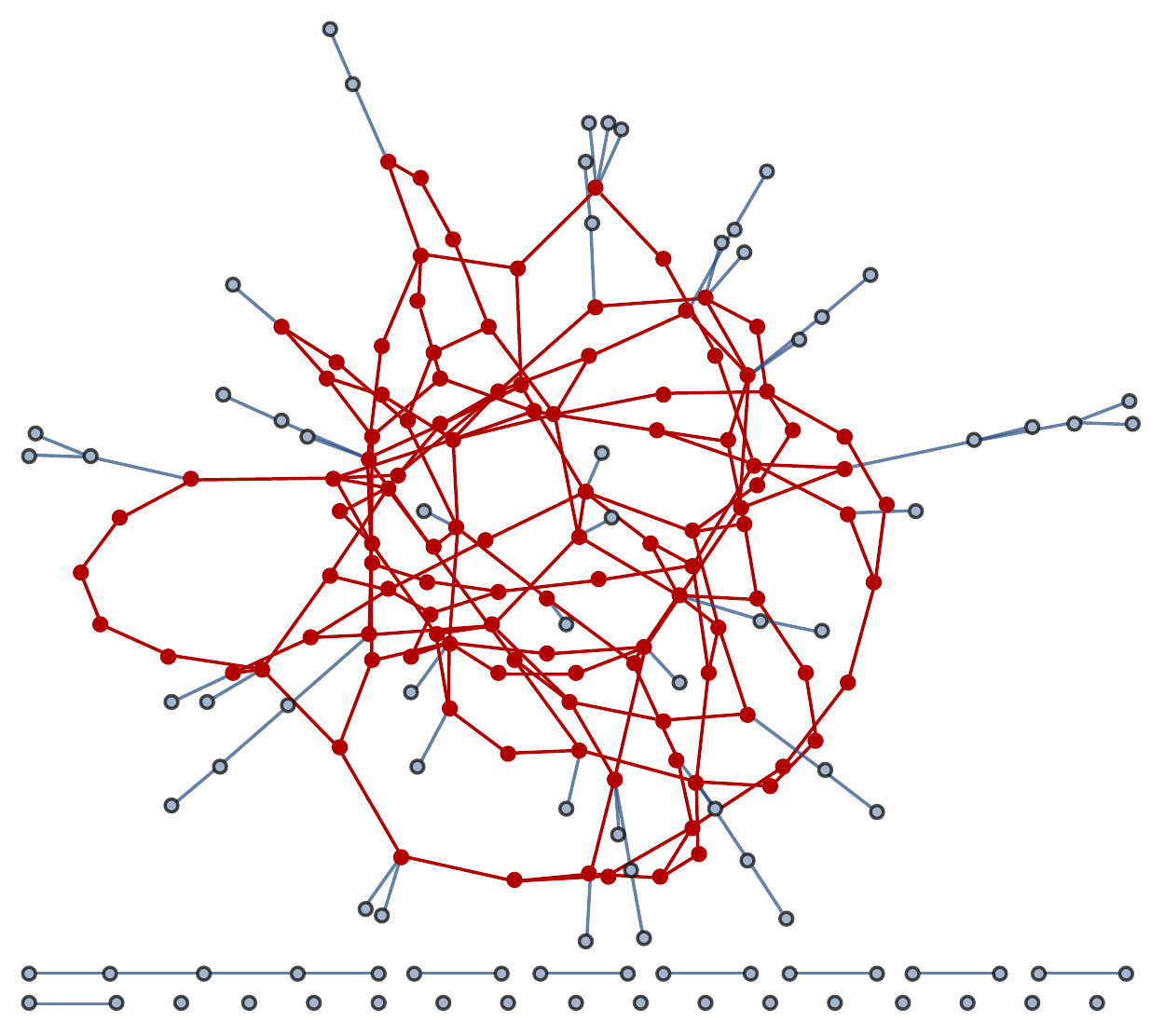}
  \caption{2-core}
  \label{fig:sfiga}
\end{subfigure}%
\begin{subfigure}{.5\textwidth}
  \centering
  \includegraphics[scale=.45]{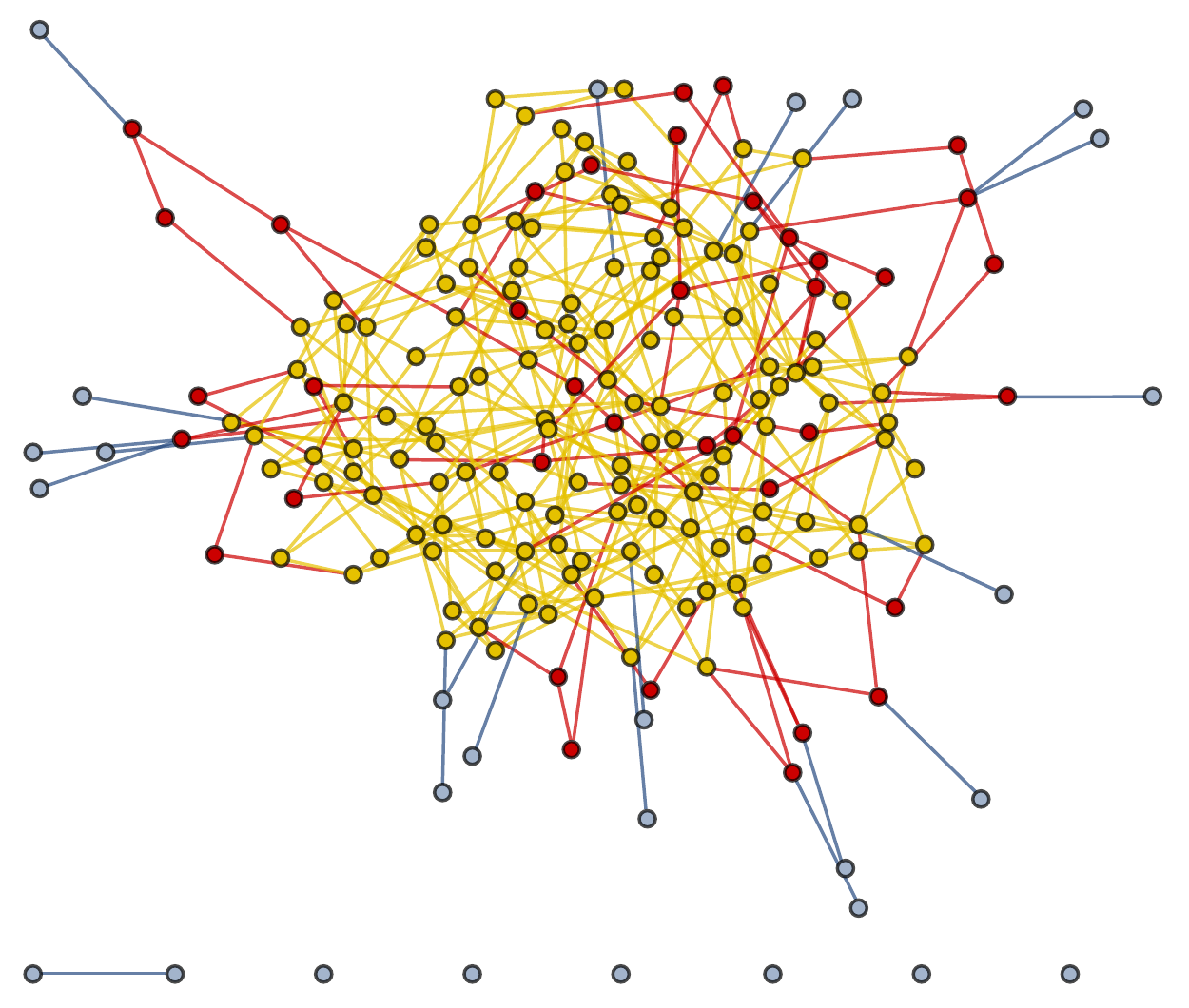}
  \caption{2-core and 3-core}
  \label{fig:sfigb}
\end{subfigure}
\caption{(a) The highlighted (red) 2-core and the trees hanging from it. (b) The yellow part highlights the 3-core, which is contained in the 2-core (union of red and yellow parts).}
\label{fig:2core}
\end{figure}

As explained above, the largest connected component $\mathscr{C}_{\sss (1)}$ of $\CM$ can be decomposed into two disjoint subsets of vertices: the 2-core $\mathscr{C}_{\sss (1)}^2$, and a forest of vertex-disjoint trees hanging from the 2-core. 
Informally speaking, the 2-core is the denser part of the graph.
Therefore, at a high level, splitting the 2-core into two parts, each containing a positive proportion of vertices, is in general \emph{costly}, and would lead to formation of a \emph{huge} number of cross edges.
Thus the optimal strategy might be to \emph{peel off} the hanging trees, since moving each hanging tree to some other partition would form precisely one cross-edge.
But in that case also, we show that the number of vertices in each of the hanging trees are 
small (essentially finite), and hence in order to move $\Theta(n)$ vertices to some other partition, $\Theta(n)$ trees must be cut, and thus, $\Theta(n)$ cross edges must be created.

To formalize the above heuristics, the proof of Proposition~\ref{prop-no-e-d-cut} breaks into two key steps, each being true with high probability:
 \begin{enumerate}[{\normalfont (i)}]
  \item The hanging trees are not heavy, in the sense that peeling off a small number of them cannot separate out a large number of vertices. This is formalized in Lemma~\ref{lem:thin-tree}.
  \item The 2-core does not have an $(\varepsilon,\delta)$-cut, which is stated in Lemma~\ref{lem-epsilon-delta-2-core}.
 \end{enumerate}
 %Most of the hanging trees are finite, i.e., by choosing $D$ large enough, the combined proportion of all the vertices lying beyond depth $D$ in any of the hanging trees can be made arbitrarily small. 
 Denote by $\mathcal{T}_h$, the set of all trees attached with the 2-core of $\mathscr{C}_{(1)}$, i.e., $T\in \mathcal{T}_h$ if and only if the subgraph in $\mathscr{C}_{(1)}$ induced by $T$ is a tree, $T\cap \mathscr{C}_{\sss (1)}^2 =\emptyset $, and there exists only one vertex $v_T\in \mathscr{C}_{\sss (1)}^2$ that shares an edge with some vertex in $T$. 
With a little abuse of notation we will write $T$ also to denote the set of vertices in $T$. 
 We always assume that each tree $T \in \mathcal{T}_h$ is rooted at the unique point $w_T$ such that $(v_T,w_T)$ is an edge and $v_T\in\mathscr{C}_{\sss (1)}^2$.
\begin{lemma}[Hanging trees are not heavy]
\label{lem:thin-tree}
For any $\varepsilon>0$, there exists $\delta=\delta_1(\varepsilon)>0$, such that with high probability, any collection $\mathcal{T}\subseteq\mathcal{T}_h$ of $\delta n$ trees contain at most $\varepsilon n$ vertices in total.
\end{lemma}
\begin{lemma}[No $(\varepsilon,\delta)$-cut in the 2-core]\label{lem-epsilon-delta-2-core}
For any $\varepsilon>0$, there exists $\delta = \delta_2(\varepsilon) > 0$ such that with high probability,
$\mathscr{C}_{\sss (1)}^2 $ does not have any $(\varepsilon, \delta)$-cut.
\end{lemma}
The proof of the above two lemmas are rather technical, and are provided at the end of the subsection. 
Now we prove Proposition~\ref{prop-no-e-d-cut} using Lemmas~\ref{lem:thin-tree} and~\ref{lem-epsilon-delta-2-core}.
In Figure~\ref{fig:structure} we provide a schematic diagram for the structure of the proof of Proposition~\ref{prop-no-e-d-cut} and the interdependence of different intermediate lemmas.

\begin{figure}
\begin{center}
\begin{tikzpicture}[line cap=round,line join=round,>=triangle 45,x=1.5cm,y=1.5cm, scale = .73]
 \node (A) at (0,-0.35) [draw,align = center, text width=4cm] {No $(\varepsilon,\delta)$-cut in $\mathscr{C}_{\sss (1)}$\\ Proposition~\ref{prop-no-e-d-cut}};
 \node (B) at (-2.5,-2.5) [draw,align = center, text width=3cm] {Hanging trees are not heavy \\ Lemma~\ref{lem:thin-tree}};
\node (C) at (2.5,-2.5) [draw,align = center, text width=3cm] {No $(\varepsilon,\delta)$-cut in 
the 2-core of $\mathscr{C}_{\sss (1)}$ \\ Lemma~\ref{lem-epsilon-delta-2-core}};
\node (D) at (-5,-5) [draw,align = center, text width=3cm] {Neighborhood of small no.~of vertices is small \\ Lemma~\ref{lem:small-nbd}};
\node (E) at (0,-5) [draw,align = center, text width=3cm] {The depth of the hanging trees are finite \\ Lemma~\ref{lem:small-hanging-trees}};
\node (F) at (-3.5,-7.75) [draw,align = center, text width=5cm] {Small no.~of vertices in intermediate components\\
Lemma~\ref{lem:small-no-giant}};
\node (G) at (5,-5) [draw,align = center, text width=3cm] {Local event approx.~of typical neighborhoods.\\ Lemma~\ref{lem:loc-event-approx}};
\node (H) at (3.5,-7.75) [draw,align = center, text width=5cm] {No cycles of short length in typical neighborhoods\\
Claim~\ref{claim:not-close-short-cycle}};

\draw[-] (B)--(C);
\draw[->] ($(B)!0.5!(C)$)--(A);

\draw[-] (D)-- (E);
\draw[->] ($(D)!0.5!(E)$)--(B);

\draw[->] (G)--(E);

\draw[-] (F)--(H);
\draw[->] (5,-6.5)--(G);
\draw[-] ($(F)!0.5!(H)$)--(0,-6.5);
%\draw[-] (5,-6.375)--(0,-6.375);
\draw[-] (5,-6.5)--(0,-6.5);

\draw[->] (5,-2.5)--(C);
\draw[-] (5,-2.5)--(G);
 \end{tikzpicture}
\end{center}
\caption{Proof structure and interdependence of different lemmas.}
\label{fig:structure}
\end{figure}
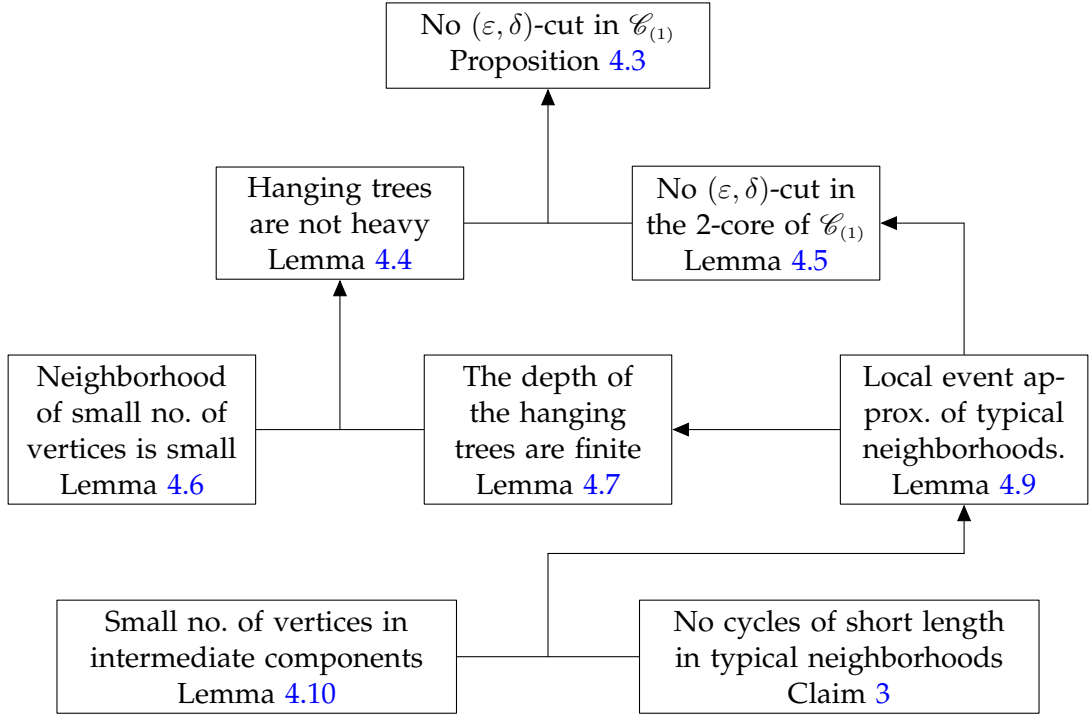

\begin{proof}[Proof of Proposition~\ref{prop-no-e-d-cut}]
Fix any $\varepsilon>0$.
Choose $\delta = \min\{\delta_1(\varepsilon/2), \delta_2(\varepsilon/2)\}$, where $\delta_1(\varepsilon)$ and $\delta_2(\varepsilon)$ are as in Lemmas~\ref{lem:thin-tree} and~\ref{lem-epsilon-delta-2-core}, respectively.

We now claim that for this choice of $\delta$, there is no $(\varepsilon,\delta)$-cut in $\mathscr{C}_{\sss (1)}$.
Indeed, existence of an $(\varepsilon,\delta)$-cut in $\mathscr{C}_{\sss (1)}$ implies that there exists $\delta n$ edges, whose removal splits $\mathscr{C}_{\sss (1)}$ into two parts, both containing at least $\varepsilon n$ vertices.
Observe that due to the choice of $\delta$, removal of any set of $\delta n$
edges can separate out at most $\varepsilon n/2$ vertices belonging to $\cup_{T\in\mathcal{T}_h}\{T\}$, and at most $\varepsilon n/2$ vertices belonging to $\mathscr{C}_{\sss (1)}^2 $  with high probability, and the proof is complete.
\end{proof}

\subsubsection{Hanging trees are not heavy}\label{sssec:notheavy} 
\begin{proof}[Proof of Lemma~\ref{lem:thin-tree}]
The proof consists of two main steps. 
The first step establishes a property of the underlying degree sequence, which states that the sum of the degrees of  `small' number of vertices is `small'.
\begin{lemma}
\label{lem:small-nbd} 
Under Assumption~\ref{assumption1}.\ref{assumption1-2}, given any $\varepsilon, r >0$, there exists $\delta = \delta(\varepsilon, r) >0$, such that for all sufficiently large $n$, the sum of degrees of the $r$-neighborhood of \emph{any} $\delta n$ vertices is at most~$\varepsilon n$, i.e. $\sum_{u\in\mathcal{N}[U,r]}d_i<\varepsilon n$ uniformly over all subsets $U\subseteq [n]$ such that $|U|< \delta n$.
\end{lemma}
In the second step we show that $r$ can be chosen large enough, so that with high probability, the total number of vertices at depth more than $r$ in all hanging trees combined, is arbitrarily `small'. This is formalized in Lemma~\ref{lem:small-hanging-trees}. 
\begin{lemma}
\label{lem:small-hanging-trees} 
For any $\varepsilon>0$, there exists $r=r(\varepsilon)>0$  such that 
with high probability 
$$\left|\mathscr{C}_{\sss (1)}\setminus\mathcal{N}[\mathscr{C}_{\sss (1)}^2,r]\right|<\varepsilon n.$$
\end{lemma}
\noindent
Given Lemmas~\ref{lem:small-nbd} and~\ref{lem:small-hanging-trees}, the proof of Lemma~\ref{lem:thin-tree} can now be completed. 
Consider the following equivalent re-statement of  Lemma~\ref{lem:thin-tree}:
\begin{quotation}
\noindent
For any $\varepsilon, \beta>0$, there exists $\delta = \delta(\varepsilon)>0$ and $n_0= n_0(\varepsilon, \beta)$, for which,
 the probability that there exists a subset $\mathcal{T}\subseteq\mathcal{T}_h$ with $|\mathcal{T}|<\delta n$ and $|\bigcup_{T\in\mathcal{T}}\{T\}|\geq \varepsilon n$, is at most $\beta$ for all $n\geq n_0$.
\end{quotation}
To show the above statement, fix any $\varepsilon, \beta>0$. 
Using Lemma~\ref{lem:small-hanging-trees}, choose $r=r(\varepsilon/2)$ and $n_1 = n_1(\varepsilon/2,\beta)$, such that for all $n\geq n_1$, 
$$\prob{\left|\mathscr{C}_{\sss (1)}\setminus\mathcal{N}[\mathscr{C}_{\sss (1)}^2,r]\right|\geq\frac{\varepsilon n}{2}}<\beta.$$
Also, appealing to Lemma~\ref{lem:small-nbd}, we choose $\delta = \delta(\varepsilon/2, r)$ and $n_2 = n_2(\varepsilon,r)$, 
such that for all $n\geq n_2$, $\nexists\ U\subseteq V$ with $|U|< \delta n$ and
$\sum_{i\in\mathcal{N}[U,r]}d_i\geq \varepsilon n/2$. 
Now observe that if there exists a subset $\mathcal{T}\subseteq\mathcal{T}_h$ with $|\mathcal{T}|<\delta n$ and $|\bigcup_{T\in\mathcal{T}}\{T\}|\geq \varepsilon n$, then 
$$\text{either}\quad\bigg|\bigcup_{T\in\mathcal{T}}\mathcal{N}[w_T,r]\bigg|>\frac{\varepsilon n}{2} \quad \text{or}\quad 
\bigg|\bigcup_{T\in\mathcal{T}}\{T\}\setminus \bigcup_{T\in\mathcal{T}}\mathcal{N}[w_T,r]\bigg|>\frac{\varepsilon n}{2},$$
where $w_T$ is the unique vertex in $T$ that has a neighboring vertex in $\mathscr{C}_{\sss (1)}^2$.
Choose $n_0 = \max\{n_1,n_2\}$ so that, for all $n\geq n_0$, the probability of the first event is 0, and that of the latter event is at most $\beta$, which concludes the proof.
\end{proof}

It remains to prove Lemmas~\ref{lem:small-nbd} and~\ref{lem:small-hanging-trees}. We start with Lemma \ref{lem:small-nbd}. 
\begin{proof}[Proof of Lemma~\ref{lem:small-nbd}]
Fix any $\varepsilon>0$.
We first verify the case when $r=1$ and prove this lemma by induction.
Due to Assumption \ref{assumption1}\ref{assumption1-2}, $K=K(\varepsilon)>0$ can be chosen such that for all sufficiently large $n$,
 \begin{equation}
 \frac{1}{n}\sum_{i\in [n]}d_i\ind{d_i>K} < \frac{\varepsilon}{2}.
 \end{equation} 
 Take $\delta = \varepsilon/ (2K)$, and fix any $V\subseteq [n]$ with $|V|< \delta n$. Then,
 \begin{equation}
  \frac{1}{n}\sum_{i\in V}d_i \leq \frac{1}{n} \sum_{i\in V}d_i\ind{d_i\leq K}+ \frac{1}{n}\sum_{i\in [n]}d_i\ind{d_i>K} < \varepsilon.
 \end{equation}
Now suppose that Lemma~\ref{lem:small-nbd} holds for some $r>0$.
Choose $\delta_1=\delta(\varepsilon,1)$, and $\delta = \delta(\delta_1,r)$. 
Notice that for any $U\subseteq [n]$ with $|U|<\delta n$,
 $|\mathcal{N}[U,r]|<\delta_1 n$, and thus, 
 $\sum_{i\in \mathcal{N}[U,r+1]}d_i=\sum_{i\in \mathcal{N}[\mathcal{N}[U,r],1]}d_i <\varepsilon n $.
\end{proof}
To prove Lemma \ref{lem:small-hanging-trees} we require a detailed understanding of the local neighborhood structure of $\CM$. 
For the ease of readability, we start with a heuristic road-map of the arguments. 
Observe that for any fixed $r>0$, and given any random observation $G$ of $\CM$, $\left|\mathscr{C}_{\sss (1)}\setminus\mathcal{N}[\mathscr{C}_{\sss (1)}^2,r]\right| = n\prob{V_n\in \mathscr{C}_{\sss (1)}\setminus\mathcal{N}[\mathscr{C}_{\sss (1)}^2,r]\given \CM=G}$, where we recall that $V_n$ denotes a typical vertex.
Therefore, it is enough to show that for any $\varepsilon>0$,
$r= r(\varepsilon)$ can be chosen large enough, such that 
$$\prob{V_n\in \mathscr{C}_{\sss (1)}\setminus\mathcal{N}[\mathscr{C}_{\sss (1)}^2,r]\given \CM= G}\pto \varepsilon'<\varepsilon, \quad \text{as}\quad n\to\infty.$$
However, it is challenging to obtain the latter probability.
For this reason, we will use the local event approximation technique, a key element in the study of sparse random graphs \cite{RGCN1,RGCN2,JLR00,BJM07,J08,BJR07}. 
In particular, our results for the configuration model mirror the ones proved in~\cite{BJR07} in the context of inhomogeneous random graphs.
Roughly speaking, the crucial idea is based upon two observations:
\begin{enumerate}[{\normalfont (i)}]
\item  The local neighborhood of a typical vertex resembles a branching process,
i.e., with high probability, the breadth-first-search (BFS) exploration starting from $V_n$ up to suitable depth can be coupled with a branching process. This is formally stated in Proposition~\ref{prop-coupling}.
\item Looking at the local neighborhood of $V_n$ up to suitable distance, it can be determined whether $V_n$ is near the 2-core.
More specifically,  the event that $V_n$ is within the $r$ neighborhood of the 2-core, is asymptotically  `equivalent' to the event that for some $L_n \to \infty$,  there exists two vertex disjoint paths of length $L_n$ from a vertex within the $r$ neighborhood of $V_n$.
This fact is later formalized in Lemma~\ref{lem:loc-event-approx}. 
\end{enumerate}
The proof follows once we have these ingredients in place. 
First we start by introducing some notations. 
Denote by $\mathcal{X}$ the branching process with initial distribution $D$ and progeny distribution $D^*-1$,
where $D$ is the limiting random variable as in Assumption~\ref{assumption1}, and $D^*$ follows the size-biased distribution of $D$, i.e., 
$$\prob{D^*=j}= \frac{j\ \prob{D=j}}{\E[D]}, \quad j\geq 1.$$
Note that the survival probability of $\mathcal{X}$ is given by $\eta$, as in~\eqref{defn:eta} (cf.~\cite{JL09}).
The number of offspring of $\mathcal{X}$ in generation $l$
is denoted by $\mathcal{Z}_l$, and  the number of vertices at distance $l$ in the breadth-first neighborhood exploration tree (i.e.~the BFS tree) starting from vertex $v$ is denoted by $Z_l(v)$.
Furthermore, define the following events:
\begin{enumerate}[{\normalfont (a)}]
\item $\mathrm{TC}_r(v)$: the vertex $v$ is within distance $r$ of the 2-core of $\mathscr{C}_{\sss (1)}$,
\item $\mathrm{LTC}_r(v,L)$: there exists a vertex $v'$ at distance $t$ of $v$, $t\leq r$, with two vertex disjoint paths of length $L$ starting at $v'$  which join $v'$ to the vertices at distance $t+L$ from $v$.
\item $\mathrm{DS}_r$:  the branching process $\mathcal{X}$ has a progeny within the first $r$ generations that has two children, both of which survive till infinity. 
\item $\mathrm{LDS}_r(L)$: the branching process $\mathcal{X}$ has a progeny within the first $r$ generations that has two children surviving further $L$ generations.
\end{enumerate}

As explained in the proof sketch above, the following proposition couples the local neighborhood of a typical vertex with the branching process $\mathcal{X}$.
\begin{proposition}[{\cite[Proposition~5.4]{RGCN2}}]
\label{prop-coupling} 
Let $\{\mathcal{Z}^1_l\}_{l\geq 1}, \{\mathcal{Z}^2_l\}_{l\geq 1}$ be two independent copies  of $\{\mathcal{Z}_l\}_{l\geq 1}$, and
 $V_n$, $W_n$ be two independent typical vertices of $\CM$. 
 There exists $(L_n)_{n\geq 1}$ such that $L_n\to \infty$, and a coupling $\big((\hat{Z}^1_l, \hat{Z}^2_l), (\hat{\mathcal{Z}}^1_l,\hat{\mathcal{Z}}^2_l)\big)_{l=1}^{L_n}$ of $\big((Z_l(V_n), Z_l(W_n)), (\mathcal{Z}^1_l,\mathcal{Z}^2_l)\big)_{l=1}^{L_n}$ such that 
 \begin{equation}\label{eq:coupling}
 \lim_{n\to\infty}\prob{\exists\ l\leq L_n: (\hat{Z}^1_l, \hat{Z}^2_l)\neq (\hat{\mathcal{Z}}^1_l,\hat{\mathcal{Z}}^2_l)}=0.
 \end{equation}
\end{proposition}
The next lemma shows that for any $(L_n)_{n\geq 1}$ that increases to infinity at a rate slower than $\log(n)$, the two events $\TC_r(V_n)$ and $\LTC_r(V_n, L_n)$ are equivalent.

\begin{lemma}\label{lem:loc-event-approx} 
Let $(L_n)_{n\geq 1}$ be such that $L_n\to\infty$ and $L_n/\log(n)\to 0$.
Then, for any fixed $r\geq 1$, 
$$\lim_{n\to\infty}\prob{\mathrm{TC}_r(V_n)\ \Delta\ \mathrm{LTC}_r(V_n,L_n)}= 0.$$
\end{lemma}
We defer the proof of Lemma~\ref{lem:loc-event-approx} until Section~\ref{sec:local_approx}, and complete the proof of  Lemma~\ref{lem:small-hanging-trees} using Lemma~\ref{lem:loc-event-approx}.
\begin{proof}[Proof of Lemma~\ref{lem:small-hanging-trees}]
Fix any $r>0$. 
Observe that for any $L_n$ such that $L_n\to\infty$,
\begin{equation*}
\lim_{n\to\infty}\prob{\LDS_r(L_n)}=\prob{\DS_r}.
\end{equation*}
Furthermore, choose $L_n^{(1)}$ according to Proposition~\ref{prop-coupling}, and $L_n^{(2)}$ such that Lemma~\ref{lem:loc-event-approx} holds.
Therefore, for $L_n = \min\{L_n^{(1)},L_n^{(2)}\}$, 
\begin{align}\label{eq:temp}
\lim_{n\to\infty} \prob{\TC_r(V_n)}
=\lim_{n\to\infty} \prob{\LTC_r(V_n,L_n)} 
=\lim_{n\to\infty} \prob{\LDS_r(L_n)}
=\prob{\DS_r}.
\end{align}
Also, 
\begin{align*}
&\frac{\left|\mathcal{N}[\mathscr{C}_{\sss (1)}^2,r]\right|}{n} = \prob{\TC_r(V_n)\given \CM}
\implies \frac{1}{n}\expt{|\mathcal{N}[\mathscr{C}_{\sss (1)}^2,r]|} = \prob{\TC_r(V_n)},
\end{align*}
and hence using~\eqref{eq:temp}, we get
\begin{equation}\label{eq:mean-conv}
\lim_{n\to\infty}\frac{1}{n}\expt{|\mathcal{N}[\mathscr{C}_{\sss (1)}^2,r]|} = \lim_{n\to\infty}\prob{\TC_r(V_n)}=\prob{\DS_r}.
\end{equation}
To find $\var{|\mathcal{N}[\C^2,r]|}$, consider two vertices $V_n$, $W_n$ chosen uniformly at random independently of the graph and independently of each other. 
Again, note that
$$\frac{|\mathcal{N}[\C^2,r]|^2}{n^2} = \prob{V_n\in\mathcal{N}[\C^2,r], W_n\in\mathcal{N}[\C^2,r]\given \CM}.$$
Thus,
\begin{equation}\label{var-negihbor-2core-d}
 \begin{split}
  \frac{1}{n^2}\expt{|\mathcal{N}[\C^2,r]|^2} &=  \prob{V_n\in\mathcal{N}[\C^2,r], W_n\in\mathcal{N}[\C^2,r]} \\
  &= \prob{\TC_r(V_n)\cap \TC_r(W_n)}. 
  \end{split}
\end{equation}
Recall from Proposition~\ref{prop-coupling} that with high probability, the $L_n$ neighborhoods of $V_n$, $W_n$ can be coupled with two independent copies of $\mathcal{X}$. 
Hence, under the given coupling
\begin{equation*}
\begin{split}
\prob{\TC_r(V_n)\cap \TC_r(W_n)}&=\prob{\LTC_r(V_n,L_n)\cap \LTC_r(W_n,L_n)}+o(1)\\
 & = \prob{\LTC_r(V_n,L_n)}\prob{\LTC_r(W_n,L_n)}+o(1)\\
 &=\prob{V_n\in\mathcal{N}[\C^2,r]}^2+o(1),
\end{split}
\end{equation*}
and it follows that
\begin{equation*}
\begin{split}
\frac{1}{n^2} \expt{|\mathcal{N}[\C^2,r]|^2} =\prob{V_n\in\mathcal{N}[\C^2,r]}^2+o(1)= \prob{\mathrm{DS}_r}^2+o(1).
  \end{split}
\end{equation*}
Therefore,
\begin{equation}\label{eq:var}
 \frac{1}{n^2}\var{|\mathcal{N}[\C^2,r]|}\to 0.
\end{equation}
Using Chebyshev's inequality, \eqref{eq:mean-conv} and \eqref{eq:var} yields for any fixed $r\geq 1$,
\begin{equation}\label{eq:r-nbd-prob}
 \frac{|\mathcal{N}[\C^2,r)]|}{n}\pto  \prob{\mathrm{DS}_r}.
\end{equation}
Now for any supercritical branching process conditioned on survival, the probability that the root has atleast two children surviving to infinity is bounded away from zero.
Therefore, conditioned on survival, the probability that any progeny in an infinite line of descendants has another child that survives till infinity is bounded away from zero. 
Thus, $\PR(\mathcal{X}\text{ survives}\setminus \mathrm{DS}_r)\leq c^r,$ for some $c<1$. 
Further, since $\mathrm{DS}_r$ is an increasing event in $r$, $\PR(\mathcal{X}\text{ survives}\setminus \cup_{r\geq 0}\mathrm{DS}_r)\leq\lim_{r\to\infty}c^r=0,$ and hence
\begin{equation}\label{eq:r-limit}
\lim_{r\to\infty}\prob{\mathrm{DS}_r} =\prob{\mathcal{X}\text{ survives}}=\eta.
\end{equation}
Using Theorem~\ref{thm:JL09}, \eqref{eq:r-nbd-prob} yields
\begin{equation}\label{eq:outside-r-nbd}
\begin{split}
\frac{\left|\mathscr{C}_{\sss (1)}\setminus\mathcal{N}[\mathscr{C}_{\sss (1)}^2,r]\right|}{n}\pto \eta - \prob{\mathrm{DS}_r}>0.
\end{split}
\end{equation}
Now, $\PR(\mathrm{DS}_r) \nearrow \eta$ as $r\to\infty$. Thus,
for any $\varepsilon>0$, we can choose $r_0 = r_0(\varepsilon)$ such that $\eta- \prob{\mathrm{DS}_r}<\varepsilon$ for all $r\geq r_0$. 
Hence, with high probability
$\left|\mathscr{C}_{\sss (1)}\setminus\mathcal{N}[\mathscr{C}_{\sss (1)}^2,r]\right|<\varepsilon n$, for all $r\geq r_0$.
\end{proof}

\subsubsection{2-core is well-connected}
\begin{proof}[Proof of Lemma~\ref{lem-epsilon-delta-2-core}]
In this proof we leverage the first moment method argument as used in~\cite{BJR07}. 
Condition on the degree sequence 
$\Mtilde{\bld{d}} = (\tilde{d}_1, \cdots, \tilde{d}_n)$ of $\mathscr{C}_{\sss (1)}^2$.  
Let $n_2 = |\mathscr{C}_{\sss (1)}^2|$ and let $m_2$  be the number of edges in the $2$-core. 

Recall that $\mathscr{C}_{\sss(1)}^2$ can be obtained from $\mathscr{C}_{\sss (1)}$ by sequentially deleting the vertices of degree one until all the vertices in the deleted subgraph have degree at least two. 
Thus, two \emph{paired} half-edges are deleted at each step, and conditional on the deleted half-edges the perfect matching on the rest of the half-edges remains a uniform perfect matching. 
In particular, $\mathscr{C}_{\sss (1)}^2$ is distributed as a configuration model conditioned on the degree sequence~$\Mtilde{\bld{d}}$ (cf. \cite[Section~3]{JL07}).
Furthermore, we will need the following estimate for the number of degree three vertices in the 2-core:
\begin{claim}\label{claim:deg-dist-2core}
Let $N_j$ denote the number vertices in the 2-core having degree $j$. Denote by $\rho_j$ the probability that the root of the branching process  $\mathcal{X}$ has exactly $j$ neighbors that survive. Then, as $n\to\infty$, $N_j/n\pto \rho_j$.
\end{claim}
\begin{claimproof}
The proof follows using similar arguments as in the proof of Lemma~\ref{lem:small-hanging-trees}. Note that $\E[N_j]/n = \PR(V_n\in \C^2, \text{ and }D_n = j)$, where $V_n$ is a typical vertex, and $D_n$ is the degree of $V_n$. 
Let $\mathrm{TS}_j(V_n)$ denote the event that $\{V_n\in \C^2, \text{ and }D_n = j\}$ and $\mathrm{LTS}_j(V_n)$ denote the (localized) event that there are $j$ disjoint non self-intersecting  paths starting from $V_n$ of length~$L_n$, where $L_n \to\infty$ such that Proposition~\ref{prop-coupling} holds. 
The essentially  same arguments as in the proof of Lemma~\ref{lem:loc-event-approx} (see Section~\ref{sec:local_approx}) can be followed to show that, for $L_n\to\infty$ and  $L_n = o(\log(n))$ ,
\begin{equation}
\PR(\mathrm{TS}_j(V_n)\ \Delta\ \mathrm{LTS}_j(V_n,L_n))\to 0.
\end{equation}
Moreover, an application of Proposition~\ref{prop-coupling} and an argument identical
to \eqref{var-negihbor-2core-d} again yields $\var{N_j}=o(n^2)$ and the proof follows.
\end{claimproof}
\noindent
Having proved the local event approximation in Section~\ref{sec:local_approx}, the rest of the proof is similar to~\cite{BJR07}, and will be sketched briefly for completeness. 

%Let us condition on the degree sequence $\Mtilde{\bld{d}}$ of $\mathscr{C}_{\sss (1)}^2$.
For any subset $A \subset \mathscr{C}_{\sss (1)}^2$, we define $\bar{A} = \mathscr{C}_{\sss (1)}^2 \backslash A$. 
Further, recall that for $A \subset \mathscr{C}_{\sss (1)}^2$, we denote the half-edges incident to the vertices in $A$ by $S(A)$. 
For a set of half-edges $S$, denote by $p(S;\Mtilde{\bld{d}})$ the probability that the half-edges of $S$ are paired among each other in $\mathscr{C}_{\sss (1)}^2$, conditional on~$\Mtilde{\bld{d}}$. Using the fact that the half-edges of $\mathscr{C}_{\sss (1)}^2$ form a uniform perfect matching conditional on the degrees,  we obtain
\begin{equation}\label{expr-p-S}
p(S;\Mtilde{\bld{d}}) = \frac{(|S| -1) !! (2m_2 - |S| - 1) !!}{(2m_2-1)!!} \leq \frac{1}{ \binom{m_2}{|S|/2 }}.
\end{equation}
A partition $A, \bar{A}$ of $\mathscr{C}_{\sss (1)}^2$ is called $(\varepsilon,\delta)$-\emph{bad} if $|A|, | \bar{A}| \geq \varepsilon n$, and  there is a subset $S\subset S(A)$ with $|S(A) \setminus S| \leq \delta n$ such that all the half-edges in $S$ are paired with each other during the random matching of the half-edges. 
Let $\Gamma_n$ denote the number of bad partitions of $\mathscr{C}_{\sss (1)}^2$.
 Thus, 
\begin{align}
\E_{\Mtilde{\bld{d}}}[\Gamma_n] \leq \sum_{ A \subset \mathscr{C}_{\sss (1)}^2 : |A|, |\bar{A}| \geq \varepsilon n } \sum_{S \subset S(A): |S(A)\backslash S| \leq \delta n} p(S;\Mtilde{\bld{d}}), \label{eq:firstmoment} 
\end{align}
where $\E_{\bld{\tilde{d}}}[\cdot]$ denotes the conditional expectation given the degree sequence of $\mathscr{C}_{\sss (1)}^2$ to be $\Mtilde{\bld{d}}$.
We need to show that for all $\varepsilon>0$, there exists $\delta = \delta(\varepsilon)>0$, such that  $\E_{\Mtilde{\bld{d}}}[\Gamma_n]\to 0$.

We first derive a lower bound on $\binom{m_2}{|S|/2}$. 
Observe that each vertex in $\mathscr{C}_{\sss (1)}^2$ has degree at least $2$, and thus, $|S(A)|/2 \geq |A|$ and $m_2 - |S(A)|/2 \geq |\bar{A}|$, where for the second inequality we have used the fact that $2m_2 -|S(A)| = S(\bar{A})$.
 Note that for a supercritical $\CM$, $\E[D(D-2)]>0$, and Assumption~\ref{assumption1}.\ref{assumption1-3} thus implies $\PR(D\geq 3) >0$.
  Therefore, by Claim~\ref{claim:deg-dist-2core}, there exists $\varepsilon_1>0$ such that the proportion of degree 3 vertices in $\mathscr{C}_{\sss (1)}^2$ is at least $\varepsilon_1 n$ with high probability. 
  Fix such an $\varepsilon_1>0$, and 
 let $\mathcal{A}_n$ denote the event that the proportion of degree 3 vertices in $\mathscr{C}_{\sss (1)}^2$ is at least $\varepsilon_1 n$.  
Note that on $\mathcal{A}_n$, one of the parts among $A$ and $\bar{A}$ contains at least $\varepsilon_1n/2$ degree three vertices. Consequently, either $|S(A) |/2 \geq |A| + \varepsilon_1 n / 4$ or $ m_2 - |S(A)|/2 \geq |\bar{A}| + \varepsilon_1 n/4$ on $\mathcal{A}_n$. 
Using these bounds, and the fact that $|A|, |\bar{A}| \geq \varepsilon n$, it follows that 
\begin{align}
\binom{m_2}{|S(A)|/2} \geq \exp(4an) \binom{n_2}{|A|} , \nonumber 
\end{align}
for some $a>0$ chosen as a function of $\varepsilon, \varepsilon_1$. Moreover, for any partitions $A,\bar{A}$ we have $|S(A) \backslash S| \leq \delta n$, and $\delta$ can be chosen small enough such that
\begin{align}
\binom{m_2}{|S|/2 } \geq \exp(3an) \binom{n_2}{|A|}, \nonumber 
\end{align}which gives the requisite lower bound.

To derive an upper bound on the number of possible choices for $A$ and $S$ in~\eqref{eq:firstmoment}, we note that given $|A|=a_0$, there are $\binom{n_2}{a_0}$ ways of choosing $A$. 
Also, given $A$, there are at most $\binom{2 m_2 }{\delta n}$ choices for $S(A)\setminus S$ such that $|S(A)\backslash S| \leq \delta n$.
Plugging these estimates back into \eqref{eq:firstmoment} yields
\begin{align}
\E_{\bld{\tilde{d}}}[\Gamma_n] \leq \sum_{ \varepsilon n \leq a_0 \leq n_2 - \varepsilon n} \binom{n_2}{a_0}  \binom{2m_2}{\delta n} \binom{n_2}{a_0}^{-1} \exp(-3an) \quad \text{on } \mathcal{A}_n. \nonumber 
\end{align}
Thus, for a small enough choice of $\delta>0$, it follows that
\begin{equation*}
 \E_{\bld{\tilde{d}}}[\Gamma_n]\to 0\quad \text{on } \mathcal{A}_n,  \quad\text{and}\quad \PR(\mathcal{A}_n^c) \to 0.
\end{equation*}
This completes the proof of Lemma~\ref{lem-epsilon-delta-2-core}. 
\end{proof} 
\subsection{Approximation of typical local neighborhoods}
\label{sec:local_approx}
%In this section, we derive detailed properties about the structures of typical neighborhoods in \CM. 
%
%
%
We prove Lemma~\ref{lem:loc-event-approx} in this section. 
A component $\mathscr{C}_{\sss (i)}$ for $i\geq 2$ (i.e., except the largest component) will be called an  \emph{intermediate component} if $|\mathscr{C}_{\sss (i)}|>L_n$ for some $L_n\to\infty$.
We need to study some structural properties of the intermediate components that will play a key role in establishing Lemma~\ref{lem:loc-event-approx}.
For any $L>0$, define
\begin{equation}
 Q_n(L):=\sum_{i\geq 2}|\mathscr{C}_{\sss (i)}|\ind{|\mathscr{C}_{\sss (i)}| \geq L}.
\end{equation}
Denote by $\mathscr{C}(v)$, the component in $\CM$ containing the vertex $v$.
%Thus, $\mathscr{C}(V_n)$ is the component in $\CM$ containing the typical vertex $V_n$
The next lemma roughly states that the proportion of vertices that belong to some
intermediate component is negligible with high probability.
\begin{lemma}[Small number of vertices in intermediate components] \label{lem:small-no-giant} For any $\varepsilon >0$, there exists $K=K(\varepsilon)$, such that
\begin{equation}\label{CVn-small}
 \limsup_{n\to\infty}\prob{V_n\notin \C, |\mathscr{C}(V_n)|>K}<\varepsilon.
\end{equation}
Consequently, for any $L_n$ such that $L_n\to\infty$, as $n\to\infty$, 
\begin{equation}\label{eq:no-giant-small}
 \frac{\expt{Q_n(L_n)}}{n}\to 0.
\end{equation}
\end{lemma}

\begin{proof}
Fix any $K\geq 1$. Recall from Theorem~\ref{thm:JL09} that 
\begin{equation}\label{giant-prob}
 \prob{V_n\in \C} =\prob{\mathscr{C}(V_n)=\C}\to \prob{|\mathcal{X}|=\infty}.
\end{equation} 
Now, based on the information about the $K$-neighborhood of $V_n$, it can be exactly determined whether the event $\{|\mathscr{C}(V_n)|>K\}$ has occurred or not.
Therefore, using Proposition~\ref{prop-coupling}, we have
\begin{equation}\label{giant-large}
 \prob{|\mathscr{C}(V_n)|\leq K} = \prob{|\mathcal{X}|\leq K}+o(1).
\end{equation} 
Combining \eqref{giant-prob} and \eqref{giant-large} yields
\begin{equation}
\begin{split}
 \prob{V_n\notin \C, |\mathscr{C}(V_n)|>K}
 =\prob{|\mathcal{X}|\in (K,\infty)}+o(1),
 \end{split}
\end{equation}
and hence \eqref{CVn-small} follows. To see \eqref{eq:no-giant-small}, notice that by \eqref{CVn-small},
\begin{equation*}
 \begin{split}
 \frac{1}{n} \expt{Q_n(L_n)}&=\E\bigg[\frac{1}{n}\sum_{i\geq 2}|\mathscr{C}_{\sss (i)}|\ind{|\mathscr{C}_{\sss (i)} |> L_n}\bigg]\\
 &=\frac{1}{n}\E\bigg[\sum_{v\in [n] }\ind{|\mathscr{C}(v) |> L_n}\bigg]=\prob{V_n\notin\C,|\mathscr{C}(V_n)|>L_n}\to 0.
 \end{split}
\end{equation*}
\end{proof}
Let us now introduce the following novel construction of the configuration model, that will allow us to relate it to the graph after deletion of one vertex. 
This will be crucial for completing the proof of Lemma~\ref{lem:loc-event-approx}. 
\begin{algo}\label{algo-CM}\normalfont 
Consider a given degree sequence $\bld{d}$ on vertex set $[n]$. Recall that $\ell_n = \sum_i d_i$ is the sum of the degrees. 
First $n_0$ isolated vertices are assigned their vertex labels. 
The algorithm below generates the random topology induced by the vertices of degree one or larger.
 \begin{enumerate}[(S1)]
  \item Initially there are $\ell_n$ degree one vertices  labeled $v(1),\dots,v(\ell_n)$, each with an attached half-edge. Call these the set of \emph{red} vertices. 
  Construct a uniform matching of these $\ell_n$ half-edges. 
  Denote the corresponding graph by $\mathcal{G}(0)$, and set $V(0)=\emptyset$.
  Also, take any permutation of the index set $\{i\in [n]:d_i>1\}$ of all vertices of degree more than one, and denote it by $\{\sigma_1,\sigma_2,\ldots,\sigma_{\hat{n}}\}$, where $\hat{n}=n-n_0-n_1$.
  \item At step $t+1$, $0\leq t\leq \hat{n}-1$, choose $d_{\sigma_t}$ degree one vertices from the graph $\mathcal{G}(t)$ uniformly at random independently of the perfect matching, and coalesce them into a single \emph{black} vertex with index $\sigma_t$. 
  Let $\mathcal{G}(t+1)$ be the new modified graph, and set $V(t+1)=V(t)\cup \{\sigma_t\}$. See Figure~\ref{fig:explosion} for an illustration of this step.
  \item After $\hat{n}^{\mathrm{th}}$ step, when all indices $i$ with $d_i>1$ are exhausted, label all the degree one vertices at random, independently of (S1) and (S2). 
 \end{enumerate}  
\end{algo}
Note that the vertex index assignment process is independent of the initial perfect matching, and therefore, at any time step $t$,  $\mathcal{G}(t)$ is a configuration model given its degree sequence.
The algorithm, thus indeed produces a configuration model with degree sequence $\bld{d}$ in the end.
This is formally stated in Lemma~\ref{lem:distn-alt-cons}.
Also,  notice that at any time step $t$,  the subgraph in $\mathcal{G}(t)$ induced by the set of black vertices remains fixed till the formation of $\CM$.
\begin{lemma}\label{lem:distn-alt-cons} 
For all $t\geq 0$,
$\mathcal{G}(t)$ is a configuration model given its degree sequence. In particular, the final graph is distributed as \CM.
\end{lemma}
\begin{remark}{\normalfont
In Algorithm~\ref{algo-CM}, the indices corresponding to the vertices with degrees at most one are assigned at the final step (S3).
It is worthwhile to note that this is not strictly necessary in order for the algorithm to work. 
In particular, since the uniform matching is created independent of the index assignments, any assignment ordering produces $\CM$ in the end.
In the proof of Lemma~\ref{lem:loc-event-approx} below, however, we will require the stated order of indexing the vertices.
}
\end{remark}

Fix any vertex $v$ of degree at least 2, and any permutation $\{\sigma_1,\sigma_2,\ldots,\sigma_{\hat{n}}\}$ of the set $\{i\in [n]:d_i>1\}$ such that $\sigma_{\hat{n}}=v$.
Denote the sequence of graphs constructed in Algorithm~\ref{algo-CM} by $\{\mathcal{G}^v(t)\}_{t\geq 0}$, i.e.,  $\mathcal{G}^v(t)$ denotes the graph at the $t^{\supth}$ step.
The $(\hat{n}-1)^{\supth}$ and $\hat{n}^{\supth}$ steps of the algorithm are schematically presented in Figure~\ref{fig:explosion}.
\begin{figure}
  \centering
  \includegraphics[scale=.5]{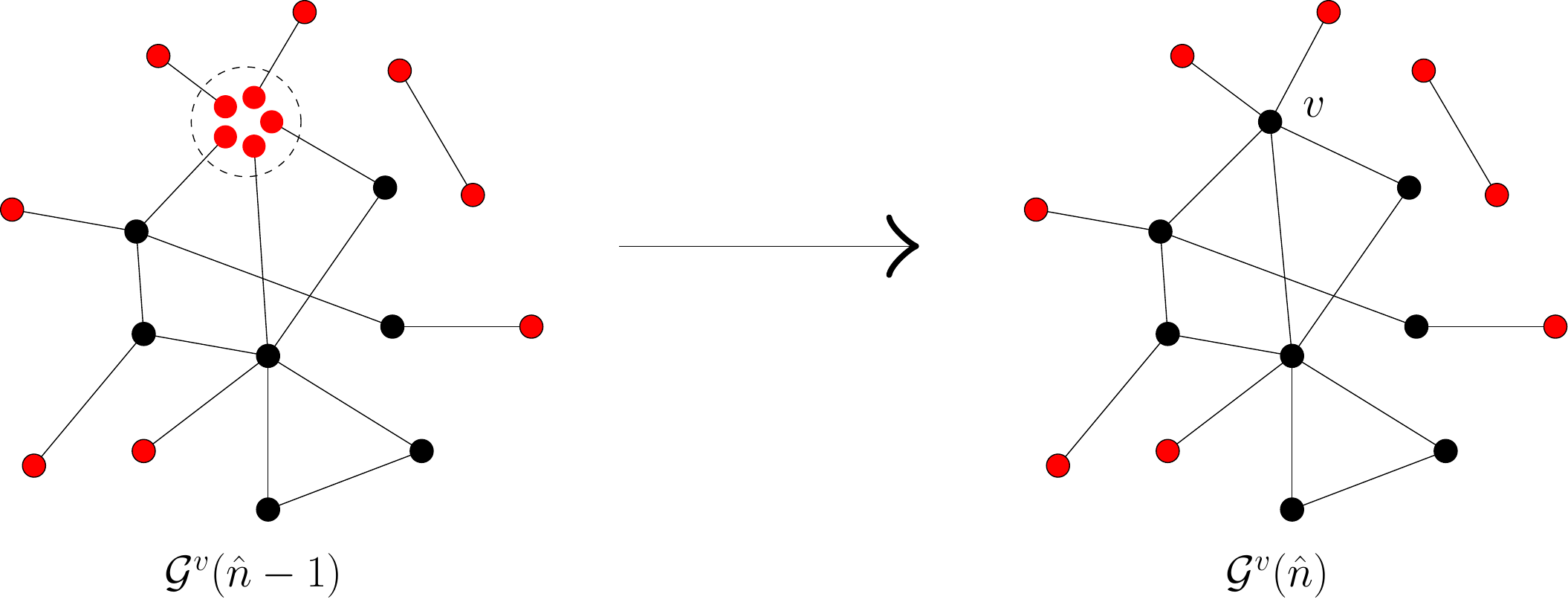}
\caption{The red vertices are the ones that have not yet been assigned any index. At $\hat{n}^{\supth}$ step, five of the unlabeled degree one vertices are selected, and the vertex $v$ is formed.}
\label{fig:explosion}
\end{figure}
Now, we complete the proof of Lemma~\ref{lem:loc-event-approx}. 
We will use the following fact:
\begin{lemma}\label{lem:maxdeg}
For any degree sequence satisfying Assumptions~\ref{assumption1}.\ref{assumption1-1},~and~\ref{assumption1}.\ref{assumption1-2}, the maximum degree $d_{\max}$ is $o(\sqrt{n})$. 
\end{lemma}
\begin{proof}
For each fixed $K\geq 1$, 
$\E[D_n^2\ind{D_n\leq K}] \to \E[D^2\ind{D\leq K}],$ 
and consequently, 
$\E[D_n^2\ind{D_n>K}] \to \E[D^2\ind{D>K}].$
Thus, 
$$\lim_{K\to\infty}\lim_{n\to\infty}\E[D_n^2\ind{D_n> K}]=
\lim_{K\to\infty}\E[D^2\ind{D>K}] =0.$$ Moreover, observe that 
\begin{equation*}
 \limsup_{n\to\infty}\frac{d_{\max}^2}{n}\leq \limsup_{n\to\infty}\Big[ \frac{1}{n}\sum_{i:d_i>K}d_i^2 + \frac{K^2}{n} \Big] = \limsup_{n\to\infty}\expt{D_n^2\ind{D_n>K}}.
\end{equation*}
Since the left side of the above inequality does not depend on $K$, it follows that
\begin{equation*}
\limsup_{n\to\infty}\frac{d_{\max}^2}{n}\leq \limsup_{K\to\infty}\limsup_{n\to\infty}\expt{D_n^2\ind{D_n>K}}=0.
\end{equation*}
\end{proof}
\begin{proof}[{Proof of Lemma~\ref{lem:loc-event-approx}}]
Fix $r>0$, and $L_n$ such that $L_n\to\infty$ and $L_n/\log(n)\to 0$ as $n\to\infty.$
The proof is split into two steps: we show that
(i) $\prob{\TC_r(V_n)\setminus \LTC_r(V_n,L_n)} \to 0$, and
(ii)~$\prob{\LTC_r(V_n,L_n)\setminus \TC_r(V_n)} \to 0$.
\paragraph*{Case-(i):} Define the event $\mathrm{C}(v,r,L)$ that the vertex $v$ is within $r$ distance from a cycle of length at most $L$.
Then note that 
$$\mathrm{TC}_r(v)\setminus \mathrm{LTC}_r(v,L)\subseteq \mathrm{C}(v,r+L,2L).$$
Indeed, suppose that $\mathrm{TC}_r(v)\setminus \mathrm{LTC}_r(v,L)$ holds. 
Let $v_1\in \mathscr{C}_{\sss (1)}^2$ be such that there exists a path $\mathcal{P}_1$ of length at most $r$ from $v$ to $v_1$ (take $v_1 = v$ if $v\in \mathscr{C}_{\sss (1)}^2$). 
Now, since $v_1$ is in the two-core, there exists at least two vertex-disjoint paths (disjoint from $\mathcal{P}_1$) starting from $v_1$, and because $\mathrm{LTC}_r(v,L)$ does not happen, any two such paths must either  meet each other, or one of them intersects itself within distance $L$ from $v'$. 
In either cases a cycle of length at most $2L$ is created that is joined to $v$ via a path of length at most $r+L$, and therefore $\mathrm{C}(v,r+L,2L)$ must hold.
\begin{claim}\label{claim:not-close-short-cycle} 
Suppose that $L_n/\log(n) \to 0$. As $n\to\infty$,  
$\prob{\mathrm{C}(V_n,L_n,L_n)}\to 0$.
\end{claim}
\begin{claimproof}
In the proof we will make use of the path counting techniques as in~\cite{J09b, BDHS17}.
 Define $\ell_n':=\ell_n-4L_n+1$. 
 Note that due to Assumption~\ref{assumption1}.\ref{assumption1-2}, a constant $\kappa >1$ can be chosen such that 
\begin{equation}\label{eq:choice-k}
\frac{1}{\ell_n'}\sum_{i\in [n]}d_i(d_i-1)\leq \kappa \quad\mbox{for all}\quad n\geq 1.
\end{equation}
 The event $\mathrm{C}(V_n,L_n,L_n)$ implies that there is a path $(V_n, x_1,x_2,\ldots, x_{l})$ of length $l \leq L_n$, and $x_{l}$ belongs to a cycle $(x_{l},x_{l+1},\ldots,x_{l+m-1})$ of length $m\leq L_n$, where the $x_i$'s are distinct.
 Fix some $V_n=v$.
Then the number of structures with a path $(v, x_1,x_2,\ldots, x_l)$ and a cycle $(x_l,x_{l+1},\ldots,x_{l+m-1})$ is given by
$$\bigg[d_v\bigg(\prod_{i=1}^{l-1}d_{x_i}(d_{x_i}-1)\bigg) d_{x_l}\bigg]\times \bigg[(d_{x_l}-1) \bigg(\prod_{i=l+1}^{l+m-1}d_{x_i}(d_{x_i}-1)\bigg)(d_{x_l}-2) \bigg],$$
where the first term in the product is due to the number of ways the path can be formed, and the second is due to the cycle.
Furthermore, each of these specific configurations has probability $[(\ell_n-1)(\ell_n-3)\dots(\ell_n-2l-2m+1)]^{-1}$.
Therefore,
 \begin{equation}
 \begin{split}
  &\prob{\mathrm{C}(V_n,L_n,L_n)\given V_n=v}\\
  & \hspace{1cm}\leq \sum_{l,m\leq L_n}\sum_{x_1,\dots,x_{l+m-1}}\dfrac{(d_{x_1}-2)d_v\prod_{i=1}^{l+m-1}d_{x_i}(d_{x_i}-1)}{(\ell_n-1)(\ell_n-3)\dots(\ell_n-2l-2m+1)}\\
  & \hspace{1cm}\leq \sum_{l,m\leq L_n}\frac{1}{(\ell_n')^{m+l}} \bigg(\sum_{i\in [n]}d_i(d_i-1)\bigg)^{l+m-2}d_v\sum_{i\in [n]}d_i(d_i-1)(d_i-2) \\
   & \hspace{1cm}\leq \sum_{l,m\leq L_n} \bigg(\frac{1}{\ell'_n}\sum_{i\in [n]}d_i(d_i-1)\bigg)^{l+m-2}d_v\frac{d_{\max}}{\ell_n'}\frac{1}{\ell_n'}\sum_{i\in [n]}d_i(d_i-1)\\
   & \hspace{1cm}\leq \sum_{l,m\leq L_n} \bigg(\frac{1}{\ell'_n}\sum_{i\in [n]}d_i(d_i-1)\bigg)^{l+m-1} d_v\frac{d_{\max}}{\ell_n'} %\\
  %& \hspace{1cm} 
\quad  \leq\  K \frac{d_v\kappa ^{2L_n}}{\sqrt{n}} 
  \end{split}
 \end{equation} 
 for some constant $K>0$ where
 in the final step we have used \eqref{eq:choice-k} and Lemma~\ref{lem:maxdeg}. 
 Therefore,
 \begin{equation}
 \prob{\mathrm{C}(V_n,L_n,L_n)}\leq \frac{K\kappa ^{2L_n}}{\sqrt{n}}\expt{D_n}= K\expt{D_n}\exp\left(2L_n\log(\kappa ) - \frac{1}{2}\log(n)\right)\to 0,
 \end{equation} 
 by Assumption~\ref{assumption1}.\ref{assumption1-2}, and the fact that $L_n=o(\log(n))$. 
\end{claimproof}
Therefore, for any fixed $r\geq 1$, $\prob{\TC_r(V_n)\setminus \LTC_r(V_n,L_n)}\leq \prob{\mathrm{C}(V_n,r+L_n,2L_n)}\leq \prob{\mathrm{C}(V_n,2L_n,2L_n)}\to 0$, and the proof of part (i) is complete.
\paragraph*{Case-(ii)}
We prove this part for $r=0$. 
The proof of the general case is included at the end.
Fix any vertex $v\in [n]$, and condition on $V_n=v$. 
If $d_v\leq 1$ or $v\notin\mathscr{C}_{\sss (1)}$, then $\prob{\LTC_0(v,L_n)} = \prob{\TC_0(v)} = 0$.
So, without loss of generality assume that $d_v>1$ and $v\in \mathscr{C}_{\sss (1)}$. 
Recall the construction in Algorithm~\ref{algo-CM} and the definition of the graph $\mathcal{G}^v(t)$.
Note that, if $\LTC_0(v,L_n)\setminus \TC_0(v)$ happens, then there are two vertex-disjoint paths in $\CM$ starting from $v$, which have length at least $L_n$, but they do not meet each other.
 Furthermore, the event $\mathrm{LTC}_0(v,L_n)\setminus \mathrm{TC}_0(v)$ is determined by the graph $\mathcal{G}^v(\hat{n})$.
 Define the event $E(v)$ that, while creating the vertex with index $v$ at time $\hat{n}$, one of the degree one vertices in one of the intermediate components of $\mathcal{G}^v(\hat{n}-1)$ was chosen.
 Observe that 
 $$\mathrm{LTC}_0(v,L_n)\setminus \mathrm{TC}_0(v)\subseteq E(v).$$
Let $Q_n^v(L_n)$ denote the total number of vertices in the intermediate components of size more than $L_n$, in the graph $\mathcal{G}^v(\hat{n}-1)$. 
Using Lemmas~\ref{lem:distn-alt-cons} and~\ref{lem:maxdeg}, it follows that $\mathcal{G}^v(\hat{n}-1)$ is a configuration model given its degree sequence that satisfy Assumption~\ref{assumption1}.
\begin{claim}\label{claim:uniform-v} $\frac{1}{n}\max_{v\in [n]}\expt{Q_n^v(L_n)}\to 0$, as $n\to\infty$.
\end{claim}
\begin{claimproof}
 Note that an application of Lemma~\ref{lem:small-no-giant} directly implies that  $n^{-1}\expt{Q_n^v(L_n)}\to 0$, for any fixed $v$. Let $(d_i^v)$ denote the degree sequence of $\mathcal{G}^v(\hat{n}-1)$ and let $\nu_n^{v}:= \sum_i d_i^v(d_i^v-1)/\sum_{i} d_i^v$. Observe that 
 (i) $\sum_id_i^v = \ell_n$, (ii) $\sum_id_i^v(d_i^v-1) = \sum_{i\in [n]}d_i(d_i-1)+O(d_{\max}^2)$. Therefore, we get (iii) $\max_{v\in [n]} |\nu_n^v-\nu_n|\to 0$ as $n\to\infty$. 
Now, while approximating the breadth-first exploration of $\mathcal{G}^v(\hat{n}-1)$ by a suitable branching process in \eqref{eq:coupling}, one can in fact obtain error estimates that are uniform over $v$. 
This is a consequence of the precise bounds stated in \cite[Lemma 5.6]{RGCN2}, that are used as the main ingredient for the proof of~\cite[Proposition~5.4]{RGCN2}.
 Therefore, while proving \eqref{giant-large} for the graph $\mathcal{G}^v(\hat{n}-1)$, one can  use (i) and (iii) above to get error estimates that are uniform in $v$. Thus, the claim follows.
\end{claimproof}
\noindent
Finally, we bound the probability of the event $E(V_n)$. 
%To this end, note that for any fixed vertex $v$, the event $E(v)$ occurs if one of the degree one vertices clumped to form $v$ is chosen from an intermediate component in $\mathcal{G}^v(\hat{n}-1)$. 
Note that, 
in $\mathcal{G}^v(\hat{n}-1)$ there are $n_1+d_v-1$ degree one vertices.
Therefore, conditional on $\mathcal{G}^v(\hat{n}-1)$, the vertex $v$ is created at step $\hat{n}$ by choosing $d_v$ vertices from a set of $n_1+d_v-1$ vertices, and $E(v)$ occurs if at least one of those degree one vertices is from an intermediate component (for which there are at most $Q_n^v(L_n)$ choices). Thus,
$$\prob{E(v)} \leq \frac{d_v}{n_1+d_v-1}\expt{Q_n^v(L_n)}$$
%
%\textcolor{red}{
%This probability is thus dominated by $\PR(W>0)$, where $W$ is a Hypergeometric random variable with population size $n_1 + d_v -1$ containing $Q_n^v(L_n)$ successes, and sample size $d_v$. }
Again,  by Assumption~\ref{assumption1}, there exists a constant $K>0$, such that $n_1+d_v-1\geq \ell_n/K$ for all large $n$. Hence, 
\begin{equation}
 \begin{split}
  \prob{E(V_n)}&=\frac{1}{n}\sum_{v\in [n]} \prob{E(v)}\leq \sum_{v\in [n]}\frac{d_v}{n_1+d_v-1}\frac{1}{n}\expt{Q_n^v(L_n)}\\
  &\leq  \frac{K}{n}\Big(\max_{v\in [n]}\expt{Q_n^v(L_n)}\Big)\sum_{v\in [n]}\frac{d_v}{\ell_n}
  \leq \frac{K}{n}\max_{v\in [n]}\expt{Q_n^v(L_n)}\to 0,
 \end{split}
\end{equation}
where the last step  follows from Claim~\ref{claim:uniform-v}.
Thus it follows that
\begin{equation}\label{local-event-eq}
\PR(\mathrm{LTC}_0(V_n,L_n)\setminus \mathrm{TC}_0(V_n))=o(1).
\end{equation}
To see the general case for $d\geq 1$, note that \eqref{local-event-eq} implies $\E[\#\{v\in [n]: \mathrm{LTC}_0(v,L)\setminus \mathrm{TC}_0(v) \text{ occurs} \}]/n\to 0$.  Using Lemma~\ref{lem:small-nbd}, it now follows that the fraction of vertices which are within the $d$ neighborhood of a vertex $v'$ for which $\mathrm{LTC}_0(v',L)\setminus \mathrm{TC}_0(v') $ occurs converges to zero in $L^1$. Therefore, $\PR(\mathrm{LTC}_d(V_n,L)\setminus \mathrm{TC}_d(V_n))=o(1)$, and the proof is complete.
\end{proof}

\section{Proof for the Max-Cut}
\label{sec:proof_maxcut}
We prove Theorem~\ref{thm:max-cut} in this section.
The proof for the sub/supercritical cases in Theorem~\ref{thm:max-cut} (i) and (ii) are provided in Sections~\ref{ssec:maxcut-sub} and~\ref{ssec:maxcut-sup}, respectively.
The case for large mean degree stated in Theorem~\ref{thm:max-cut} (iii) is proved in Section~\ref{ssec:large-mean}  

\subsection{Subcritical case}\label{ssec:maxcut-sub}
The idea in the subcritical regime is to count the number of cycles.
This idea has also been adopted in the proof of~\cite[Theorem 19]{CGHS04} for Erd\H{o}s-R\'enyi random graphs.
Observe that the bipartite components (components with no cycles or only cycles of even length) contribute all of their edges to the Max-Cut.
To analyze the non-bipartite components we first observe in Lemma~\ref{lem:all-unicyclic-comp} that all the components of a subcritical $\CM$ are unicyclic (contains only one cycle) with high probability. 
%
%
%
%Also, if there is an odd cycle, then the max-cut must leave one edge uncut.
%It is thus of prime importance to count the number of cycles. 
%But one challenge might be caused by overlapping cycles, since then it becomes difficult to count the number of uncut edges.
%This challenge is overcome by a known result, stated in Lemma~\ref{lem:all-unicyclic-comp} below. 
%It states that in the sub-critical regime, with high probability there are no such overlapping cycles in $\CM$.
%
%
\begin{lemma}[{\cite[Theorem~1.2~(b)]{HM2012}}] \label{lem:all-unicyclic-comp}
For subcritical $\CM$ satisfying {\rm Assumption~\ref{assumption1}}, the probability that there exists a component with more than one cycle tends to zero as $n\to\infty$.
\end{lemma}
Observe that  the Max-Cut leaves precisely one edge uncut in each of these unicyclic, non-bipartite components.
Therefore, the number of uncut edges in the Max-Cut is with high probability equal to the number of cycles of odd length that the graph contains.
Now, the asymptotic number of cycles of length $k$ in $\CM$, for any fixed $k\geq 1$, is derived in~\cite[Theorem 2.18]{B2012}, and is stated in the following lemma.
Let $C_k^n$ denote the number of cycles of length $k$ in $\CM$ (a cycle of length one denotes a loop and of length two denotes a multiple edge).
\begin{lemma}[{\cite[Theorem 2.18]{B2012}}] \label{lem:poi-cycle} 
Consider $\CM$ satisfying {\rm Assumption~\ref{assumption1}}. Then, for any $K\geq 1$, as $n\to\infty$,
\begin{equation}
 (C_k^n)_{k\in [K]} \dto (X_k)_{k\in [K]},
\end{equation}where $X_k \sim \mathrm{Poisson} (\nu^k/2k)$, independently for $k\in [K]$.
\end{lemma}
The next lemma proves that with high probability, there are no cycles of growing length.
This will be used to show that asymptotically, the total number of odd-length cycles is equal to the sum of
the number of all cycles of finite and odd length.

\begin{lemma} \label{lem:l1-tight-cycle} Consider a subcritical $\CM$ satisfying {\rm Assumption~\ref{assumption1}}. Then, 
\begin{equation}
 \lim_{K\to\infty}\lim_{n\to\infty} \prob{\exists\ k>K: C_k^n\geq 1 } = 0.
\end{equation}
\end{lemma}
Lemma~\ref{lem:l1-tight-cycle} is proved at the end of this subsection. 
Now, we prove result for the subcritical Max-Cut by using Lemmas~\ref{lem:all-unicyclic-comp}, \ref{lem:poi-cycle} and~\ref{lem:l1-tight-cycle}

\begin{proof}[Proof of Theorem~\ref{thm:max-cut}~(i)] 
As mentioned earlier, the Max-Cut leaves precisely one edge uncut in each of the unicyclic, non-bipartite components, and by Lemma~\ref{lem:all-unicyclic-comp}, with high probability, the total number of uncut edges precisely equals to the total number of odd-length cycles.
Therefore, recalling that the total number of edges equals $\ell_n/2$, it follows that $ \ell_n/2-  \Mcut (\CM)= \sum_{k\geq 1, k \text{ is odd}}C_k^n$, with high probability.
Hence, Lemmas~\ref{lem:poi-cycle} and~\ref{lem:l1-tight-cycle} yield, as $n\to\infty$,
\begin{equation*}
\frac{\ell_n}{2}- \Mcut (\CM)=\sum_{k\geq 1, k \text{ is odd}}C_k^n \dto X,\quad  X\sim \mathrm{Poisson}\bigg(\sum_{\substack{k\geq 1,\\k \text{ is odd}}}\frac{\nu^k}{2k}\bigg).
\end{equation*} 
\end{proof}

\begin{proof}[{Proof of Lemma~\ref{lem:l1-tight-cycle}}]
For brevity of notation, denote by $M$ the total number of edges, i.e., $M=\ell_n/2$.
We find the expected value of $C_k^n$ using again the path-counting techniques.
 To this end, we first fix $k$ distinct vertices $x_1,\dots,x_k$ which participate in the cycle in the given order. We denote by $\mathcal{I}_k = \{(x_1,\dots,x_k):x_i\neq x_j,\ \forall i\neq j\}$. For each vertex $x_i$, the two half-edges which participate in the cycle may be chosen in $d_{x_i} ( d_{x_i}-1)$ ways. The number of ways to pair these half-edges is thus $\prod_i d_{x_i}(d_{x_i}-1)$. 
% However, this argument counts each cycle $2k$ times, depending on $k$ different starting points and two different orientations. Thus we divide this count by $2k$ to get the correct number of pairings containing this cycle. 
%Note that the remaining  $2M-2k$) half-edges can be paired among themselves in ${2M \choose M} \frac{M!}{2^M}$ (resp. ${2M-2k \choose M-k} \frac{(M-k)!}{2^{M-k}}$) ways. 
For any fixed $R\geq 1$, $2R$ half-edges can be paired among each other in $\binom{2R}{R} R!/2^R$. 
Therefore,
\begin{equation}
\begin{split}
\mathbbm{E}[C_k^n]
&=\sum_{\mathcal{I}_k}\prod_{i=1}^k d_{x_i}(d_{x_i}-1)\frac{ \binom{2M-2k}{ M-k}\frac{(M-k)!}{2^{M-k}}}{\binom{2M}{M}\frac{M!}{2^{M}}}\\
&\leq \bigg(\frac{1}{(n)_k}\sum_{\mathcal{I}_k}\prod_{i=1}^k d_{x_i}(d_{x_i}-1)\bigg)\frac{2^k(n)_k(M)_k}{(2M)_{2k}}\\
&\leq\bigg(\frac{1}{n}\sum_{i\in [n]} d_i(d_i-1)\bigg)^k\frac{2^k(M)_k}{(n)_k(2M)_{2k}},
\end{split}
\end{equation}where the last step follows from \cite[Theorem 52]{HLP53} (see also the proof of \cite[Lemma 5.1]{J09c}).
Now, using Stirling's formula we have,
$$(n)_k= \exp[k\ln (n)-k^2/2n-O(k/n+k^3/n^2)].$$
In analogy with \eqref{eq:nu}, we define $\nu_n := \mathbbm{E}[D_n(D_n-1)]/\mathbbm{E}[D_n]$. 
Therefore,
\begin{align}
&\mathbbm{E}[\#\mbox{ cycles in }\mathrm{CM}_n(\boldsymbol{d}) \text{ of lengths in } (K, \sqrt{n})]=\sum_{k=K+1}^{\sqrt{n}}\mathbbm{E}[C_k^n] \nonumber \\
&\leq \kappa_1\sum_{k=K+1}^{\sqrt{n}} \nu_n^k\bigg(\frac{1}{n}\sum_{i\in [n]} d_i\bigg)^k 2^k\exp\bigg(k\ln (n)-\frac{k^2}{2n}+k\ln( M)-\frac{k^2}{2M}-2k\ln (2M)+\frac{k^2}{M}\bigg) \nonumber \\   
&\leq \kappa_1\sum_{k=K+1}^{\sqrt{n}} \nu_n^k\exp\bigg\{-\frac{k^2}{2}\Big(\frac{1}{n}-\frac{1}{M}\Big)\bigg\}, \nonumber
\end{align}where the constant $\kappa_1>0$ can be chosen to be independent of $K$.
Now, for subcritical $\mathrm{CM}_n(\boldsymbol{d})$, we have $M<n$. To see this, note that by the Cauchy-Schwarz inequality,
  \begin{equation}
   M=\frac{1}{2} \sum_{i\in [n]}d_i \leq \frac{\sqrt{n}}{2}\sqrt{\sum_{i\in [n]}d_i^2}= \frac{\sqrt{n}}{2}\sqrt{2M(1+\nu_n)}.
  \end{equation}
  Taking the square on both sides and using the fact that since $\nu_n< 1$, we get
  \begin{align*}
  M\leq \frac{n}{2}(1+\nu_n)< n.
  \end{align*}
  Therefore, $(1/M-1/n)>0$, and hence,
  $$\max_{k\leq \sqrt{n}}\exp\left\{-\frac{k^2}{2}\left(\frac{1}{n}-\frac{1}{M}\right)\right\}\leq \exp\left\{-\frac{n}{2}\left(\frac{1}{n}-\frac{1}{M}\right)\right\}\leq \kappa_2, $$where the constant $\kappa_2>0$ is independent of $K$. 
  %The minimum of the above quantity is also lower bounded by some other constant $C'>0$. 
  Thus,
  \begin{equation}\label{cycle-k-sqrt-n}
  \mathbbm{E}(\#\mbox{ cycles in }\mathrm{CM}_n(\boldsymbol{d}) \text{ of length in }  (K,\sqrt{n}))=\kappa_1\kappa_2\sum_{k=K+1}^{\infty}\nu_n^k\to 0,
  \end{equation}if we first take $n\to\infty$ and then $K\to\infty$.
  To count the number of cycles of length $>\sqrt{n}$, note that
  \begin{equation}\label{cycle:geq-sqrtn}
   \begin{split} 
   &\prob{\exists\text{ a cycle of length more than } \sqrt{n}}\leq \prob{\exists i\geq 1: |\mathscr{C}_{(i)}|>\sqrt{n}}= \prob{|\mathscr{C}_{\max} |>\sqrt{n}}.
   \end{split}
  \end{equation}  
Now, an application of \cite[Theorem 1.3]{J08} yields that $|\mathscr{C}_{\max}| = \OP(d_{\max})=o(n^{1/2})$ and therefore the probability in \eqref{cycle:geq-sqrtn} tends to 0 as $n\to \infty$. 
The proof of Lemma~\ref{lem:l1-tight-cycle} is now complete by combining \eqref{cycle-k-sqrt-n}, and \eqref{cycle:geq-sqrtn}.
\end{proof}
\subsection{Supercritical case}\label{ssec:maxcut-sup}

The proof for the supercritical case builds upon the following idea: 
the fact that a graph has small Max-Cut implies that deletion of a small number of edges can make the graph bipartite.
When the graph is supercritical, deletion of a small number of edges can still leave it supercritical.
In that case, if one can show that the  probability of the latter supercritical graph being bipartite is small, then  the original supercritical cannot have a small Max-Cut.

This idea has been leveraged in~\cite[Theorem 21]{CGHS04} to prove the phase-transition of Max-Cut result for the Erd\H{o}s-R\'enyi random graph.
The main challenge of implementing this idea for the configuration model is that if a set of edges is deleted from a configuration model (possibly depending on the outcome of the random graph topology), then the edge-deleted graph is not distributed as a configuration model given its degree sequence.
It thus becomes challenging to approximate the probability that after a number of edge deletion the graph becomes bipartite.
Inspired by the above issues, in case of $\CM$ we introduce a notion of \emph{blowing up} vertices.
In a way, \emph{blowing up} a vertex is the reverse process of forming a vertex at (S2) of Algorithm~\ref{algo-CM}.
Let $G=(V,E)$ be any graph. 
Also let $v\in V$ be a vertex of degree $d_v\geq 2$ with $\{u_1,u_2,\ldots,u_{d_v}\}$ being the set of neighbors in $G$.
Then define the graph $G_{b}(v)$ as follows: replace $v$ by a collection of $d_v$ degree one vertices $\{v_1,v_2,\ldots,v_{d_v}\}$, and for $i=1,2,\ldots,d_v$, add the edge $(u_i,v_i)$.
We say that $G_b(v)$ is obtained by blowing up the vertex $v$.
The graph obtained by blowing up a set of vertices $U\subseteq V$ each with degree at least 2 is defined as sequentially blowing up each vertex in $U$.
Just like the edge deletion, note that if a graph has small Max-Cut, then by blowing up a small number of vertices it should be possible to make the graph bipartite.
Now, it is crucial to note that for any set of vertices $U\subseteq V$ each with degree at least 2, the graph $G_b(U)$ is distributed as a configuration model given its degree sequence.
The above key observation enables us to estimate the probability that blowing up a small set of vertices makes the graph bipartite.

Thus our proof argument builds in two steps as follows: 
(i) First in Lemma~\ref{lem:prob-bipartite} we show that the probability that a supercritical configuration model is bipartite is exponentially small, and then 
(ii) Using union bound we establish that for any $\nu>1$, there exists a $\delta>0$, for which the probability that blowing up \emph{any} set of $\delta n$ vertices makes the graph bipartite converges to 0. 
This will complete the proof of  Theorem~\ref{thm:max-cut}~(ii).
First we formally state and prove Lemma~\ref{lem:prob-bipartite}.

Notice that since $\nu>1$ and $\PR(D=1)> 0$, we must have some $k\geq 2$ such that $\PR(D=k)>0$. 
Without loss of generality, in the rest of this section we assume that $\PR(D=2)>0$. The argument below remains identical when $\PR(D=2)=0$, in which case we proceed with $\min\{k: \PR(D=k) >0\}<\infty$ instead of 2. 
Recall that $\hat{n} = n-n_0-n_1$.
Denote $\hat{n}^* = \hat{n}-n_2$.
Note that in Algorithm~\ref{algo-CM} until the time step $\hat{n}^*$, first the vertices of degree larger than 2 are formed.
After this the vertices of degree 2 are formed during time steps $\hat{n}^*+1\leq t\leq \hat{n}$, followed by creating vertices of degree one for $t>\hat{n}$. 
It is crucial to observe that for $t(\varepsilon) = \hat{n} - \varepsilon \ell_n$, the graph $\mathcal{G}_n(t(\varepsilon))$ is distributed as a configuration model with the criticality parameter
\begin{equation}
\nu_n(\varepsilon) =\frac{\sum_{i\in [n]}d_i(d_i-1)-2\varepsilon \ell_n}{\sum_{i\in [n]}d_i} =\nu_n -2\varepsilon, \quad \text{and} \quad \lim_{n\to\infty}\nu_n(\varepsilon) = \nu(\varepsilon)>1,
\end{equation}
for $\varepsilon>0$ sufficiently small.
Denote by $\mathscr{C}_{\sss (1)}(t)$ the largest connected component of $\mathcal{G}_n(t)$.
Then by Theorem~\ref{thm:JL09}~(i), for $t\geq t(\varepsilon)$ there exists $\eta (t)>0$ such that 
\begin{equation}\label{eq:giant-G-eps}
\frac{|\sCL(t)|}{n} \pto \eta(t). 
\end{equation}
\begin{lemma}\label{lem:prob-bipartite}
There exists a constant $C_0>0$, such that 
\begin{equation}\label{eq:prob-bip-exp}
\prob{\CM \text{ is bipartite}}\leq \e^{-C_0 n}.
\end{equation}  
\end{lemma}
\begin{proof}
First note that it is enough to show 
\begin{equation}\label{eq:prob-bip-exp-2}
\prob{\CM \text{ is bipartite} \vert \mathcal{G}_n(t(\varepsilon))}\leq \e^{-C_0 n(1+\varepsilon_n)},
\end{equation}
for some $\varepsilon_n \geq 0$ almost surely.
Recall Algorithm~\ref{algo-CM}.
Also, observe that if $\sCL(t(\varepsilon))$ is non-bipartite, then $\CM$ will also be non-bipartite. 
Indeed, if $\sCL(t(\varepsilon))$ is non-bipartite, then it must contain an odd cycle of black nodes, and the process of merging degree one (red) vertices does not affect these existing cycles. 
Thus if there is an odd-length cycle in $\sCL(t(\varepsilon))$, that cycle will be present in $\CM$ as well. 
So assume $\sCL(t(\varepsilon))$ is bipartite.

For $t\geq t(\varepsilon)$ we will now describe an algorithm for partitioning $\sCL(t)$ into two vertex-disjoint sets $H_1(t)$ and $H_2(t)$ in a coupled way.
The sets $H_1(t)$ and $H_2(t)$ are such that if $\sCL(t)$ is bipartite, then these are the unique partite sets.
For $i=1,2$, let $H_i^B(t)$ and $H_i^R(t)$ be the set of black and red vertices in $H_i(t)$, respectively. Also, let $R(t)$ and $R_P(t)$ denote the set of all red vertices and red pairs in $\cG_n(t)$, respectively
(recall that a pair is two degree one vertices joined with each other).
 With a little abuse of notation, we will also write $H_i^B(t)$, $R(t)$ etc.~to denote the cardinality of the respective sets.
\begin{algo}\label{algo:partition}
\normalfont
Initially consider the unique bipartition of $\sCL(t(\varepsilon))$:
$$\sCL(t(\varepsilon)) = H_1(t(\varepsilon))\sqcup H_2(t(\varepsilon)),\ \mathrm{say}.$$ 
At time step $t> t(\varepsilon)$, suppose two red vertices $v_1$ and $v_1$ are coalesced to form a new black vertex $v$ of degree 2. 
Then the sets $H_1(t)$ and $H_2(t)$ are updated according to the following rule:
\begin{enumerate}[{\normalfont (i)}]
\item If both $v_1, v_2\in H_i^R(t-1)$ for either $i=1$ or 2, then 
$$H_i(t)= \big(H_i(t-1)\setminus \{v_1,v_2\}\big)\cup \{v\},$$
and the other partition remains unchanged.
\item If $v_1\in H_1^R(t-1)$ and $v_2\in R(t-1)\setminus (H_2^R(t-1)\cup H_1^R(t-1))$, then 
\begin{align*}
H_1(t) &= \Big(H_1(t-1)\setminus \{v_1\}\Big)
\cup\{v\}\bigcup_{k=1}^\infty \Big(\mathcal{N}[v_2,2k]\setminus \mathcal{N}[v_2,2k-1]\Big),\\  
H_2(t) &= H_2(t-1)\bigcup_{k=0}^\infty \Big(\mathcal{N}[v_2,2k+1]\setminus \mathcal{N}[v_2,2k]\Big)
\end{align*}
where by convention, the zero neighborhood of a vertex is the vertex itself.
\item If $v_1\in H_2^R(t-1)$ and $v_2\in R(t-1)\setminus (H_2^R(t-1)\cup H_1^R(t-1))$, then repeat Step (ii) above by interchanging the role of $H_1$ and $H_2$.
\item If both $v_1, v_2\in R(t-1)\setminus (H_2^R(t-1)\cup H_1^R(t-1))$, then $H_i(t) = H_i(t-1)$, $i=1,2$.
\item If $v_1\in H_1^R(t-1)$ and $ v_2\in H_2^R(t-1)$ or vice versa, then remove $v_1$ and $v_2$ from their respective sets, and add $v$ to $H_1(t)$ or $H_2(t)$ arbitrarily. 
\end{enumerate}
\end{algo}
\vspace{.25cm}

\noindent
Note that for $t_1<t_2$, $H_1^B(t_1)\subseteq H_1^B(t_2)$ and $H_2^B(t_1)\subseteq H_2^B(t_2)$.
Also, if $\CM$ is bipartite, then $\cG_n(t)$ must be bipartite for all $t>t(\varepsilon)$.
In that case, Case (v) of Algorithm~\ref{algo:partition} should not occur, otherwise the bipartiteness will be lost. 
We now claim that the number of red vertices in both partitions will become order $n$ at some time step after $t(\varepsilon)$.
Afterwards we establish that if the claim is true, then the probability that
Case (v) of Algorithm~\ref{algo:partition} will not occur is exponentially small.
As a result, $\CM$ is bipartite with exponentially small probability.
\begin{claim}\label{claim:good-cut}
For sufficiently small $\varepsilon>0$, there exists $\delta_1=\delta_1(\varepsilon)>0$ and $\delta_2=\delta_2(\varepsilon)>0$ such that
\begin{equation}
\prob{\forall\ t>t(\varepsilon),\ H_1^R(t)<\delta_1n\  \mbox{\rm or }H_2^R(t)<\delta_2n\vert\mathcal{G}_n(t(\varepsilon))} \leq \e^{-C_0n (1+\oP(1))}.
\end{equation}
\end{claim}
\begin{claimproof}
Note that since $\mathcal{G}_n(t(\varepsilon))$ is distributed as a configuration model,  \cite[Theorem~2.3~(i)]{JL09} and Lemma~\ref{lem:num-pairs} respectively yields
\begin{equation}\label{eq:total-red-giant}
\frac{H_1^R(t(\varepsilon))+H_2^R(t(\varepsilon))}{n} \pto r(\varepsilon)>0\qquad\mathrm{and}\qquad \frac{R_P(t(\varepsilon))}{n}\pto r_P(\varepsilon)>0.
\end{equation}
Due to~\eqref{eq:total-red-giant}, one of the sets $H_1^R(t(\varepsilon))$ or $H_2^R(t(\varepsilon))$ has atleast $r(\varepsilon) n/4$ red vertices with high probability.
Without loss of generality, let the part be $H_1^R(t(\varepsilon))$.
Since at each time step $t> t(\varepsilon)$ only degree 2 vertices are created, the set of red vertices can deplete by at most 2. 
Therefore, for $t^*(\varepsilon) = t(\varepsilon)+\min\{r(\varepsilon) n/8,\varepsilon\ell_n/2\}$, 
\begin{equation}\label{eq:delta1up}
\inf_{t(\varepsilon)\leq t\leq t^*(\varepsilon)} H^R_1(t)\geq\frac{r(\varepsilon)n}{8}=:\delta_1n.
\end{equation}
It is important to note that both $r(\varepsilon)$ and $r_P(\varepsilon)$ are bounded away from zero as $\varepsilon\to 0$. Thus $\delta_1$ remains positive even when $\varepsilon$ is chosen small enough.
Recall from Algorithm~\ref{algo:partition}(ii) that during the time interval $[t(\varepsilon), t^*(\varepsilon)]$, $H_2^R$ increases by at least one if some red vertex in $H_1^R$ is coalesced with one of the vertices outside the set $H_1^R\cup H_2^R$, in particular, with one belonging to some red pair in $R_P$. 
Using \eqref{eq:delta1up}, at each time step the probability of the latter event, conditionally on $\mathcal{G}_n(t(\varepsilon))$,  is atleast 
$$\frac{r(\varepsilon)n}{8}r_P(\varepsilon)n{n_1 + 2 n_2\choose 2}^{-1}\geq \delta_1 c_1(1+\oP(1)),$$
for some $c_1\in (0,1]$.
Denote by $\mathcal{A}(t)$ the cumulative number of red vertices thus added to $H_2^R$ up to time $t$ starting from $t(\varepsilon)$.
 Observe that $\mathcal{A}(t)$ stochastically dominates a binomial random variable with $\min\{r(\varepsilon) n/8,\varepsilon\ell_n/2\}$ number of trials and the success probability atleast $\delta_1 c_1(1+\oP(1))$. 
Therefore, standard concentration inequalities for the binomial distribution \cite[Corollary 2.3]{JLR00} yields 
\begin{equation}\label{eq:bound-deg-1-A}
\PR(\mathcal{A}(t^*(\varepsilon))\leq \delta_2' n|\mathcal{G}_n(t(\varepsilon)))\leq \e^{-C_0'n (1+\oP(1))}
\end{equation}
for some suitable $\delta_2'>0$ and a constant $C_0'>0$.
Further notice that some of the red vertices in $H_2^R(t)$ have been coalesced to form new black vertices during $[t(\varepsilon),t^*(\varepsilon)]$.
This can occur if only if at the coalescence step both red vertices are selected from $H_2^R$, which occurs with probability at most ${H_2^R(t)\choose 2}/{n_1\choose 2}$.
Denote the cumulative number of such coalesced red vertices up to time $t$ by $\mathcal{B}(t)$. 
Define 
$$\tau(\varepsilon):= \inf\{t\geq t(\varepsilon): H_2^R(t) \geq  \delta_2'n/2\}\wedge t^*(\varepsilon).$$
Then observe that for $t\in [t(\varepsilon), \tau(\varepsilon)]$, the quantity $\mathcal{B}(t)$ is dominated by a binomial random variable  with $\min\{r(\varepsilon) n/8,\varepsilon\ell_n/2\}$ number of trials and success probability at most $\delta_2' n/2n_1$. 
The mean of this binomial random variable is of the order at most $\varepsilon\delta_2'n$, which can be made arbitrarily small compared to $\delta_2'n$ by choosing $\varepsilon$ small enough.
Therefore, standard concentration inequalities for the binomial distribution again imply that 
\begin{equation}\label{eq:bound-deg-1-B}
\prob{\mathcal{B}(\tau(\varepsilon))\geq \delta_2'n/4\ \vert\  \mathcal{G}_n(t(\varepsilon)), \{H_1^R(t) \leq  \delta_2'n,\forall\ t(\varepsilon) \leq t < \tau(\varepsilon)\}}\leq \e^{-C_0''n(1+\oP(1))}
\end{equation}
for some constant $C_0''>0$.
Now set $\delta_1$ as above and $\delta_2 := \delta_2'/2$, and observe that $H_2^R(t) = H_2^R(t(\varepsilon))+\mathcal{A}(t)-2\mathcal{B}(t)$. 
Thus \eqref{eq:bound-deg-1-A} and \eqref{eq:bound-deg-1-B} yields that either the probability that the following will not occur is exponentially small
\begin{align*}
H_2^R(\tau(\varepsilon)) &= H_2^R(t(\varepsilon))+\mathcal{A}(\tau(\varepsilon))-2\mathcal{B}(\tau(\varepsilon))\geq \delta_2 n,
\end{align*}
or $H_2^R(t)>\delta_2n$ for some $t\in[t(\varepsilon),  \tau(\varepsilon)]$.
%and the probability that neither of the above two events occur is exponentially small. 
In either of the two cases, this completes the proof of Claim~\ref{claim:good-cut}.
\end{claimproof}
\noindent
Let $t^{**}(\varepsilon):=\tau(\varepsilon)+(\delta_1\wedge\delta_2)n/4$.
Observe that due to Claim~\ref{claim:good-cut} and the argument given above~\eqref{eq:delta1up},
\begin{equation}
\inf_{\tau(\varepsilon)\leq t\leq t^{**}(\varepsilon)} H^R_i(t)\geq \frac{\delta_in}{4}\qquad i = 1,2.
\end{equation}
Recall that while forming the degree 2 vertices from $\mathcal{G}_n(\tau(\varepsilon))$ to $\CM$, the bipartiteness is lost if one red vertex is chosen from~$H_1^R$ and the other is chosen  from~$H_2^R$.
Therefore, for any $t\in[\tau(\varepsilon), t^{**}(\varepsilon)]$,
\begin{align*}
&\prob{\text{Bipartiteness is not preserved at time }t \vert \mathcal{G}_n(t(\varepsilon)),\ H_1^R(\tau(\varepsilon))\geq \delta_1n,\  H_2^R(\tau(\varepsilon))\geq \delta_2n}\\
&\geq \frac{\delta_1\delta_2n^2}{16{n_1 \choose 2}} = (1-c_0)(1+\oP(1)),
\end{align*}
for some $c_0\in (0,1)$.
Thus the probability that bipartiteness is preserved troughtout  the time interval $[\tau(\varepsilon), t^{**}(\varepsilon)]$ given $\mathcal{G}_n(t(\varepsilon))$, $H_1^R(\tau(\varepsilon))\geq \delta_1n$,  and $H_2^R(\tau(\varepsilon))\geq \delta_2n$ is upper bounded by $\exp(-C_0(1+\oP(1))n)$
for some constant $C_0>0$.
The proof of Lemma~\ref{lem:prob-bipartite} is now complete.
\end{proof}
We need one further lemma to complete the proof of Theorem~\ref{thm:max-cut}~(ii).
\begin{lemma}
\label{lem:small-sq-nbd} 
Under Assumption~\ref{assumption1}.\ref{assumption1-2}, given any $\varepsilon>0$, there exists $\delta = \delta(\varepsilon) >0$, such that for all sufficiently large $n$, the sum of squares of the degrees of \emph{any} collection of $\delta n$ vertices is at most~$\varepsilon n$, i.e. $\sum_{u\in U}d_i(d_i-1)<\varepsilon n$ uniformly over all subsets $U\subseteq [n]$ such that $|U|< \delta n$.
\end{lemma}
The proof of Lemma~\ref{lem:small-sq-nbd} is identical to the argument given in the proof of Lemma~\ref{lem:small-nbd}, and hence is omitted. With all the above ingredients in place, we now proceed to prove Theorem~\ref{thm:max-cut}~(ii).
\begin{proof}[Proof of Theorem~\ref{thm:max-cut}~(ii)]
Fix $\varepsilon>0$.
Using Lemma~\ref{lem:small-sq-nbd}, let us choose $\delta_0= \delta_0(\varepsilon)>0$ such that $|U|\leq \delta_0n$ implies $\sum_{i\in U}d_i(d_i-1)<\varepsilon n$.
Notice that if we blow up at most $\delta_0 n$ vertices, then the criticality parameter of the new graph changes by at most $\varepsilon n/\ell_n$, and thus for small $\varepsilon>0$, the blown up graph is also supercritical with high probability. 
Further, let $\mathcal{E}_k$ denote the event that the distance from bipartiteness of $\CM$ is $k$.
Thus, choosing $\delta<\delta_0$ small enough,  Lemma~\ref{lem:prob-bipartite} yields
\begin{equation}
\prob{\mathrm{DistBip}(\CM)\leq \delta n}\leq \sum_{k=1}^{\delta n} {\ell_n/2 \choose k} \PR(\mathcal{E}_k)\leq \sum_{k=1}^{\delta n} {\ell_n/2 \choose k}\e^{-C_0 n} \to 0.
\end{equation} 
This completes the proof of Theorem~\ref{thm:max-cut}~(ii).
\end{proof}

\subsection{High-density regime}\label{ssec:large-mean}
The proof for the supercritical case (Theorem~\ref{thm:max-cut}~(iii)) uses the first moment method. 
For a set of vertices $A$, recall that $S(A)$ is the total number of half-edges associated with $A$, and $E(A,A^c)$ is the number of edges between $A$ and $A^c$. 
Partition the graph $\CM$ in two parts $A$ and $A^c$, where we assume without loss of generality that $S=S(A)\leq \ell_n/2$. %We note that without loss of generality, any
% partition can be represented in this form. 
In this case, 
 \begin{equation}
 \label{eq:upper_bound}
  \begin{split}
  \prob{E(A,A^c)=K}& = \frac{\binom{S }{K}K! (S-K-1)!! \binom{\ell_n-S }{ K}(\ell_n-S-K-1)!!}{(\ell_n-1)!!} \\
  &= \exp\big(\ell_n f(x_n, y_n)  (1+ o(1))\big) , \quad \text{where }\\
 f(x, y)  &= \ln\big(x^{x} (1-x)^{1-x}(1-x-y)^{-(1-x-y)/2}(x-y)^{-(x-y)/2}y^{-y}\big),
  \end{split}
 \end{equation}
 and $y_n = K/ \ell_n$, $x_n = S/ \ell_n$.  
 We note that 
 \begin{align}
 \frac{\partial f}{\partial x} (x,y) = \ln \bigg(\frac{x}{1-x}\bigg) - \frac{1}{2} \ln \bigg(\frac{x-y}{1- x -y}\bigg). \nonumber 
 \end{align}
 The fact that $x/(1-x) > (x-y)/(1-x-y)$ for any $0<y< x \leq 1/2$, implies that for any fixed $y$, $\{f(x,y): x \leq 1/2\}$ is maximized at $x=1/2$. 
Now, for any $1\leq K\leq \ell_n/2$,
 \begin{equation}
  \Big|\Big\{U\subset V: \sum_{i\in U}d_i=K\Big\}\Big|\leq 2^n.
 \end{equation}  
 Thus, for any constant $c>0$, a union bound and \eqref{eq:upper_bound} yields 
 \begin{align}
 &\PR\Big(\Mcut(\CM) \geq \frac{\ell_n}{4} + n c \sqrt{\mu}\Big) = \PR\Big(\exists\ A \subseteq[n] , S(A) \leq \frac{\ell_n}{2} , E(A,A^c) \geq \frac{\ell_n}{4} + n c \sqrt{\mu} \Big)  \nonumber\\
 &\hspace{1cm}\leq \sum_{S=1}^{\ell_n/2} \sum_{A: S(A) = S} \sum_{K=\frac{\ell_n}{4}+ n c \sqrt{\mu}}^{S} \PR\Big( E(A,A^c) = K \Big) \nonumber\\ 
 %
%& \hspace{1cm}\leq 2^n \frac{\ell_n}{2} \max_{ \frac{\ell_n}{4} + n c \sqrt{\mu} \leq K \leq H}  \PR\Big(S \subseteq[n] , H(S) \leq \frac{\ell_n}{2} , E(S,S^c)=K \Big) \nonumber\\
%
%&\hspace{1cm}\leq  2^n \Big(\frac{\ell_n}{2} \Big)^2 
%\max_{ \frac{\ell_n}{4} + n c \sqrt{\mu} \leq K \leq H}
%\max_{1\leq H \leq \frac{\ell_n}{2}} \PR\Big(S\subseteq [n], H(S) = H, E(S,S^c) = K\Big) \nonumber\\
%
%&\hspace{1cm}\leq  2^n \Big(\frac{\ell_n}{2} \Big)^2  
% \max_{ \frac{\ell_n}{4} + n c \sqrt{\mu} \leq K \leq \frac{\ell_n}{2}}
% \max_{ K \leq H \leq \frac{\ell_n}{2}} \exp \Big[\ell_n f \Big(\frac{H}{\ell_n}, \frac{K}{\ell_n} \Big) \Big]\nonumber\\
%
&\hspace{1cm}\leq  2^n \Big(\frac{\ell_n}{2} \Big)^2 \max_{ \frac{\ell_n}{4} + n c \sqrt{\mu} \leq K \leq \frac{\ell_n}{2}}  \exp \Big[\ell_n f \Big(\frac{1}{2}, \frac{K}{\ell_n}  \Big) (1+ o(1))\Big] . \nonumber 
 \end{align}
Writing $K/\ell_n = y$,  notice that
 \begin{equation}
  \begin{split}
   &2^n\exp [\ell_n f(1/2,y)] \\
   &=\exp\bigg[n\ln(2)+\ell_n\ln\bigg(\frac{1}{2}y^{-y}\Big(\frac{1}{2}-y\Big)^{-\frac{1}{2}(\frac{1}{2}-y)}\Big(\frac{1}{2}-y\Big)^{-\frac{1}{2}(\frac{1}{2}-y)}\bigg)+o(n)\bigg]\\
   &= \exp\bigg[(1+o(1))\ell_n\ln\bigg(2^{1/\mu}\Big(\frac{1}{2}\Big(\frac{1}{2}-y\Big)^{-\frac{1}{2}+y}y^{-y}\Big)\bigg)\bigg],
  \end{split}
 \end{equation}
since $\ell_n/n\to\mu$. 
Therefore, we obtain
\begin{equation*}
\begin{split}
&\PR\Big(\Mcut(\CM) \geq \frac{\ell_n}{4} + n c \sqrt{\mu}\Big)\\
&\hspace{1cm}\leq \kappa n^2\max_{ \frac{1}{4} + \frac{c}{\sqrt{\mu}} \leq y \leq \frac{1}{2}} \exp\bigg[\ell_n\ln\Big(2^{\frac{1}{\mu}-1}\Big(\frac{1}{2}-y\Big)^{-\frac{1}{2}+y}y^{-y}\Big)\bigg],
\end{split}
\end{equation*}for some constant $\kappa>0$.
Now observe that $\big(\frac{1}{2}-y\big)^{-\frac{1}{2}+y}y^{-y}$ is non-increasing in the interval $\big(1/4,1/2]$, and therefore the above maximum is attained at $y=  \frac{1}{4} + \frac{c}{\sqrt{\mu}}$.
Define 
$$f(c,\mu):=\Big(\frac{1}{4}-\frac{c}{\sqrt{\mu}}\Big)^{-\left(\frac{1}{4}-\frac{c}{\sqrt{\mu}}\right)}\Big(\frac{1}{4}+\frac{c}{\sqrt{\mu}}\Big)^{-\left(\frac{1}{4}+\frac{c}{\sqrt{\mu}}\right)}-2^{-\frac{1}{\mu}+1} $$
and
\begin{equation}\label{eq:cstarmu}
c^\star(\mu)=\inf_{c>0}\Big\{c: f(c,\mu)<0 \Big\}.
\end{equation}
Thus, we can conclude that for any $c>c^\star(\mu)$,
$$\prob{\frac{1}{n}\Mcut(\CM)>\frac{\mu}{4}+c\sqrt{\mu}}\to 0,\quad \text{as}\quad n\to\infty.$$
%\vspace{0.25cm}

\noindent
Note that $f(0,\mu) = 2$, $\lim_{c\to (\sqrt{\mu}/4)^-}f(c,\mu) = \sqrt{2}$, and $f(\cdot,\mu)$ in strictly decreasing. Therefore, $c^\star(\mu)<\sqrt{\mu}/4$ for any $\mu>2$.
To see that $c^\star(\mu)\nearrow \sqrt{\ln(2)}/2$ as $\mu\nearrow\infty$, it can be checked using Taylor expansion with respect to $c$ around 0 that
$f(c,\mu) \approx 2(1-2^{-1/\mu})-8c^2/\mu$.
Also, for large $\mu$, $1-2^{-1/\mu}\approx \ln(2)/\mu$.
Thus for large $\mu$, the value of $c^\star(\mu)$ is given by $\sqrt{\ln(2)}/2$.
\qed

{%\small
\section*{Acknowledgment}
The authors sincerely thank Remco van der Hofstad for several helpful discussions.
The authors also thank Sem Borst and Remco van der Hofstad for a careful reading of the manuscript.
SD and DM were financially supported by The Netherlands Organization for Scientific Research (NWO) through Gravitation Networks grant 024.002.003, and DM was also supported by TOP-GO grant 613.001.012. SS was partially supported by the William R.~and Sara Hart Kimball Stanford Graduate Fellowship.

%\small
\bibliographystyle{apa}
\bibliography{cut}

\begin{thebibliography}{}

\bibitem[\protect\astroncite{Albert and Barab\'asi}{2002}]{AB02}
Albert, R. and Barab\'asi, A.~L. (2002).
\newblock {Statistical mechanics of complex networks}.
\newblock {\em Reviews of Modern Physics}, 74(1):47--97.

\bibitem[\protect\astroncite{Barab\'asi and Albert}{1999}]{BA99}
Barab\'asi, A.~L. and Albert, R. (1999).
\newblock {Emergence of scaling in random networks}.
\newblock {\em Science}, 286:509--512.

\bibitem[\protect\astroncite{Bertoni et~al.}{1997}]{BCP97}
Bertoni, A., Campadelli, P., and Posenato, R. (1997).
\newblock An upper bound for the maximum cut mean value.
\newblock In {\em Graph-Theoretic Concepts in Computer Science: 23rd
  International Workshop}, pages 78--84, Berlin, Heidelberg. Springer.

\bibitem[\protect\astroncite{Bhamidi et~al.}{2017}]{BDHS17}
Bhamidi, S., Dhara, S., van~der Hofstad, R., and Sen, S. (2017).
\newblock {Universality for critical heavy-tailed random graphs: Metric
  structure of maximal components}.
\newblock {\em arXiv:1703.07145}.

\bibitem[\protect\astroncite{Bollob{\'{a}}s}{1980}]{B80}
Bollob{\'{a}}s, B. (1980).
\newblock {A probabilistic proof of an asymptotic formula for the number of
  labelled regular graphs}.
\newblock {\em European J. Combin.}, 1(4):311--316.

\bibitem[\protect\astroncite{Bollob{\' a}s}{2001}]{bollobas}
Bollob{\' a}s, B. (2001).
\newblock {\em Random Graphs}.
\newblock Cambridge University Press, 2 edition.

\bibitem[\protect\astroncite{Bollob\'{a}s et~al.}{2007}]{BJR07}
Bollob\'{a}s, B., Janson, S., and Riordan, O. (2007).
\newblock The phase transition in inhomogeneous random graphs.
\newblock {\em Random Structures and Algorithms}, 31(1):3--122.

\bibitem[\protect\astroncite{Bordenave}{2012}]{B2012}
Bordenave, C. (2012).
\newblock Lecture notes on random graphs and probabilistic combinatorial
  optimization.

\bibitem[\protect\astroncite{Britton et~al.}{2007}]{BJM07}
Britton, T., Janson, S., and Martin-L\"{o}f, A. (2007).
\newblock Graphs with specified degree distributions, simple epidemics, and
  local vaccination strategies.
\newblock {\em Advances in Applied Probability}, 39(4):922--948.

\bibitem[\protect\astroncite{Chang and Du}{1987}]{CD87}
Chang, K.~C. and Du, D. H.~C. (1987).
\newblock Efficient algorithms for layer assignment problem.
\newblock {\em IEEE Transactions on Computer-Aided Design of Integrated
  Circuits and Systems}, 6(1):67--78.

\bibitem[\protect\astroncite{Chen et~al.}{1983}]{CKC83}
Chen, R.-W., Kajitani, Y., and Chan, S.-P. (1983).
\newblock A graph-theoretic via minimization algorithm for two-layer printed
  circuit boards.
\newblock {\em IEEE Transactions on Circuits and Systems}, 30(5):284--299.

\bibitem[\protect\astroncite{Coppersmith et~al.}{2004}]{CGHS04}
Coppersmith, D., Gamarnik, D., Hajiaghayi, M., and Sorkin, G.~B. (2004).
\newblock {Random MAX SAT, random MAX CUT, and their phase transitions}.
\newblock {\em Random Structures and Algorithms}, 24(4):502--545.

\bibitem[\protect\astroncite{Daud{\'{e}} et~al.}{2012}]{DMRR12}
Daud{\'{e}}, H., Mart{\'{i}}nez, C., Rasendrahasina, V., and Ravelomanana, V.
  (2012).
\newblock {The MAX-CUT of sparse random graphs}.
\newblock In {\em Proceedings of the twenty-third annual ACM-SIAM symposium on
  Discrete Algorithms}, pages 265--271. Society for Industrial and Applied
  Mathematics.

\bibitem[\protect\astroncite{D\'{i}az et~al.}{2002}]{diazpetitserna}
D\'{i}az, J., Petit, J., and Serna, M. (2002).
\newblock A survey of graph layout problems.
\newblock {\em ACM Computing Surveys (CSUR)}, 34(3):313--356.

\bibitem[\protect\astroncite{Faloutsos et~al.}{1999}]{FFF99}
Faloutsos, M., Faloutsos, P., and Faloutsos, C. (1999).
\newblock {On power-law relationships of the Internet topology}.
\newblock {\em ACM SIGCOMM Computer Communication Review}, 29(4):251--262.

\bibitem[\protect\astroncite{Feige and Krauthgamer}{2002}]{FK02}
Feige, U. and Krauthgamer, R. (2002).
\newblock {A polylogarithmic approximation of the minimum bisection}.
\newblock {\em SIAM Journal on Computing}, 31(4):1090--1118.

\bibitem[\protect\astroncite{Fountoulakis}{2007}]{F07}
Fountoulakis, N. (2007).
\newblock {Percolation on sparse random graphs with given degree sequence}.
\newblock {\em Internet Mathematics}, 4(1):329--356.

\bibitem[\protect\astroncite{Halperin and Zwick}{2001}]{HZ01}
Halperin, E. and Zwick, U. (2001).
\newblock A unified framework for obtaining improved approximation algorithms
  for maximum graph bisection problems.
\newblock In {\em Proceedings of the 8th International IPCO Conference on
  Integer Programming and Combinatorial Optimization}, pages 210--225, London,
  UK. Springer-Verlag.

\bibitem[\protect\astroncite{Hansen and Mladenovi{\'{c}}}{2002}]{HM02}
Hansen, P. and Mladenovi{\'{c}}, N. (2002).
\newblock {\em Developments of Variable Neighborhood Search}, pages 415--439.
\newblock Springer US.

\bibitem[\protect\astroncite{Hardy et~al.}{1952}]{HLP53}
Hardy, G.~H., Littlewood, J.~E., and P{\'{o}}lya, G. (1952).
\newblock {\em {Inequalities}}.
\newblock Cambridge University Press.

\bibitem[\protect\astroncite{H{\aa}stad}{2001}]{H01}
H{\aa}stad, J. (2001).
\newblock {Some optimal inapproximability results}.
\newblock {\em Journal of the ACM}, 48(4):798--859.

\bibitem[\protect\astroncite{Hatami and Molloy}{2012}]{HM2012}
Hatami, H. and Molloy, M. (2012).
\newblock The scaling window for a random graph with a given degree sequence.
\newblock {\em Random Structures and Algorithms}, 41(1):99--123.

\bibitem[\protect\astroncite{Hofstad}{2014}]{RGCN2}
van der Hofstad, R. (2017).
\newblock {\em {Random Graphs and Complex Networks.}}, volume~II.
\newblock To appear in Cambridge Series in Statistical and Probabilistic Mathematics, 2017.

\bibitem[\protect\astroncite{Hofstad}{2016}]{RGCN1}
van der Hofstad, R. (2017).
\newblock {\em {Random Graphs and Complex Networks.}}, volume~I.
\newblock Cambridge Series in Statistical
and Probabilistic Mathematics, 2017.

\bibitem[\protect\astroncite{Janson}{2008}]{J08}
Janson, S. (2008).
\newblock {The largest component in a subcritical random graph with a power law
  degree distribution}.
\newblock {\em The Annals of Applied Probability}, 18(4):1651--1668.

\bibitem[\protect\astroncite{Janson}{2009a}]{J09}
Janson, S. (2009a).
\newblock {On percolation in random graphs with given vertex degrees}.
\newblock {\em Electronic Journal of Probability}, 14:87--118.

\bibitem[\protect\astroncite{Janson}{2009b}]{J09c}
Janson, S. (2009b).
\newblock {The probability that a random multigraph is simple}.
\newblock {\em Combinatorics, Probability and Computing}, 18(1-2):205--225.

\bibitem[\protect\astroncite{Janson}{2010}]{J09b}
Janson, S. (2010).
\newblock {Susceptibility of random graphs with given vertex degrees}.
\newblock {\em Journal of Combinatorics}, 1(3-4):357--387.

\bibitem[\protect\astroncite{Janson et~al.}{1993}]{JKLP93}
Janson, S., Knuth, D.~E., {\L}uczak, T., and Pittel, B. (1993).
\newblock {The birth of the giant component}.
\newblock {\em Random Structures and Algorithms}, 4(3):233--358.

\bibitem[\protect\astroncite{Janson and Luczak}{2007}]{JL07}
Janson, S. and Luczak, M.~J. (2007).
\newblock {A simple solution to the k-core problem}.
\newblock {\em Random Structures and Algorithms}, 30(1-2):50--62.

\bibitem[\protect\astroncite{Janson and Luczak}{2009}]{JL09}
Janson, S. and Luczak, M.~J. (2009).
\newblock {A new approach to the giant component problem}.
\newblock {\em Random Structures and Algorithms}, 34(2):197--216.

\bibitem[\protect\astroncite{Janson et~al.}{2000}]{JLR00}
Janson, S., {\L}uczak, T., and Rucinski, A. (2000).
\newblock {\em {Random Graphs.}}
\newblock Wiley, New York.

\bibitem[\protect\astroncite{Jerrum and Sorkin}{1993}]{JS93}
Jerrum, M. and Sorkin, G.~B. (1993).
\newblock Simulated annealing for graph bisection.
\newblock In {\em Proceedings of 34th Annual Symposium on Foundations of
  Computer Science, 1993}, pages 94--103. IEEE.

\bibitem[\protect\astroncite{Khot}{2004}]{K04}
Khot, S. (2004).
\newblock {Ruling out PTAS for graph min-bisection, densest subgraph and
  bipartite clique}.
\newblock In {\em 45th Annual IEEE Symposium on Foundations of Computer
  Science}, pages 136--145. IEEE.

\bibitem[\protect\astroncite{Luczak and McDiarmid}{2001}]{LM01}
Luczak, M.~J. and McDiarmid, C. (2001).
\newblock Bisecting sparse random graphs.
\newblock {\em Random Structures and Algorithms}, 18(1):31--38.

\bibitem[\protect\astroncite{Mezard et~al.}{1987}]{MPV87}
Mezard, M., Parisi, G., and Virasoro, M. (1987).
\newblock {\em Spin Glass Theory and Beyond, An Introduction to the Replica
  Method and Its Applications}.
\newblock World Scientific.

\bibitem[\protect\astroncite{Molloy and Reed}{1995}]{MR95}
Molloy, M. and Reed, B. (1995).
\newblock {A critical-point for random graphs with a given degree sequence}.
\newblock {\em Random Structures and Algorithms}, 6(2-3):161--179.

\bibitem[\protect\astroncite{Poljak and Tuza}{1995}]{poljak1995tuza}
Poljak, S. and Tuza, Z. (1995).
\newblock Maximum cuts and large bipartite subgraphs.
\newblock {\em DIMACS Series}, 20:181--244.

\bibitem[\protect\astroncite{Siganos et~al.}{2003}]{SFFF03}
Siganos, G., Faloutsos, M., Faloutsos, P., and Faloutsos, C. (2003).
\newblock {Power laws and the AS-level internet topology}.
\newblock {\em IEEE/ACM Transactions on Networking}, 11(4):514--524.

\end{thebibliography}

}
\end{document}